\documentclass[letterpaper,11pt,reqno]{amsart} 
\usepackage[portrait,margin=3cm]{geometry}
\usepackage{systeme}
\usepackage{mathrsfs,xfrac} 
\usepackage[colorlinks=true,linkcolor=blue,citecolor=blue,urlcolor=blue]{hyperref} 
\usepackage{amsmath,amssymb,amsthm,amsfonts,amsbsy,latexsym,dsfont,color} 
\usepackage[numeric,initials,nobysame]{amsrefs} 
\usepackage[foot]{amsaddr}

\usepackage[utf8]{inputenc}
\usepackage[english]{babel}
\usepackage{comment}

\usepackage[textsize=tiny]{todonotes}
\usepackage{regexpatch}
\makeatletter
\xpatchcmd{\@todo}{\setkeys{todonotes}{#1}}{\setkeys{todonotes}{inline,#1}}{}{}
\makeatother

\usepackage{enumerate}
\newenvironment{enumeratei}{\begin{enumerate}[\upshape (i)]}{\end{enumerate}}

\renewcommand{\le}{\leqslant} 
\renewcommand{\ge}{\geqslant} 
\renewcommand{\leq}{\leqslant} 
\renewcommand{\geq}{\geqslant}

\newcommand{\ind}{\mathds{1}}
\newcommand{\eps}{\varepsilon}

\DeclareMathOperator{\E}{\mathds{E}}
\DeclareMathOperator{\pr}{\mathds{P}}

\newcommand{\hh}{\hat{h}}
\DeclareMathOperator{\Br}{Br}
\DeclareMathOperator{\degstate}{st}
\makeatletter
\renewcommand{\paragraph}[1]{\if@noskipsec\leavevmode\fi\par\addvspace{.5\baselineskip plus .2\baselineskip}\noindent{\normalfont\bfseries #1}\hspace{.5em}\ignorespaces}
\makeatother

\newtheorem{theorem}{Theorem}[section]

\newtheorem{proposition}[theorem]{Proposition}
\newtheorem{lemma}[theorem]{Lemma}
\newtheorem{remark}[theorem]{Remark}

\usepackage{graphicx}

\usepackage{algorithm}
\usepackage{algpseudocode}
\usepackage{etoolbox}
\AtBeginEnvironment{algorithmic}{\small}
\algrenewcommand{\algorithmicrequire}{\textbf{Input:}}
\begin{document}

\title[Approximate counting for mixed $p$-spin glasses]{Approximate Counting for Mixed $p$-Spin Glasses in Weak External Fields via Cluster Expansion}
\author[Wu]{Qiang Wu}
\address{Department of Mathematics, 1291 University of Oregon, 175 Prince Lucien Campbell Hall, Eugene, OR 97403-1205, USA}
\email{qiangwu@uoregon.edu}
\date{\today}
\thanks{}
\begin{abstract}
We give algorithms for approximating the partition function~$Z_N$ of the Sherrington--Kirkpatrick (SK) model and finite mixed Ising~$p$-spin models with $p\ge 3$. The algorithms apply at a fixed margin below the second-moment threshold~$\beta_{\mathrm f}$ with homogeneous weak external fields~$|h|\leq\rho_0N^{-1/4}$, for fixed~$\rho_0>0$. For any accuracy~$\varepsilon\in(0,1)$ and failure parameter~$\delta\in(0,1)$, they approximate~$\log Z_N$ to additive error~$\varepsilon$. The number of exact operations is polynomial jointly in~$N$, $\varepsilon^{-1}$, and~$\delta^{-1}$. The disorder variables have i.i.d symmetric distribution with unit variance. The failure probability is at most~$\delta$ plus an explicit disorder tail bound. This tail bound tends to zero under a finite fourth moment condition. Our approach builds on the cluster expansions of Aizenman--Lebowitz--Ruelle~\cite{ALR87} and Dey--Wu~\cites{DW23,DW26}. For higher-order mixtures, we retain hypergraphs using few vertices and evaluate their total weight deterministically through local spin averages and inclusion--exclusion. For SK, we impose cutoffs on the number of vertices of degree at least three and on the total number of edges. We estimate the retained sum using randomized color coding~\cite{AYZ95}. A lower tail control based on reverse hypercontractivity~\cite{MOS13} converts the mean-square approximation of the normalized partition function into logarithmic accuracy.
\end{abstract}

\maketitle
\setcounter{tocdepth}{1}\tableofcontents

\section{Introduction and main results}\label{sec:introduction}

Spin glasses are models in statistical physics described by Gibbs measures with random spin interactions. Determining their limiting free energy has been a central problem for decades. Talagrand~\cite{Tal06} and Panchenko~\cite{panchenko2014parisi} proved the Parisi formula for the Sherrington--Kirkpatrick (SK) model and its $p$-spin variants, which gives a variational characterization of this limit. The algorithmic counterpart is instead asking how to compute the partition function of a finite system in polynomial time approximately or exactly. For the classical Ising model and related spin systems, this problem has been studied extensively, see a non-exhaustive list in~\cites{JS93,LLY13,PR17,LSS19}. For spin glasses, this algorithmic problem has only been studied recently. Gamarnik and K{\i}z{\i}lda{\u{g}}~\cite{GK21} established average case hardness results for ``exact" evaluation of the SK partition function. Several recent works present various (quasi)-polynomial time approximation algorithms in high temperature regimes~\cites{BHL25,HLMM26,SZ26,HLY26}. Here we develop an alternative approach based on cluster expansion to this problem, we defer discussions on its relation to other works in Section~\ref{subsec:related-work} and~\ref{subsec:concurrent-work}.

\subsection{Definitions}\label{subsec:model}

Let~$\mathcal P$ be a finite, nonempty set of spin interaction orders, with coefficients~\mbox{$a_p>0$} for~$p\in\mathcal P$. An interaction of order~$p$ involves~$p$ different sites. We consider the SK model, for which~$\mathcal P=\{2\}$ and~$a_2=1$, and higher-order mixtures, for which~$\mathcal P\subseteq\{3,4,\ldots\}$. Denote~$p_m=\min\mathcal P$ and~$P=\max\mathcal P$.

For~$N\geq P$, let~$[N]=\{1,\ldots,N\}$ and~$\Sigma_N=\{-1,1\}^N$ be the spin configuration space. The interactions of order~$p$ are the~$p$-element subsets~$\mathcal E_{N,p}=\binom{[N]}p$. Set
\begin{equation*}
\mathcal E_N=\bigcup_{p\in\mathcal P}\mathcal E_{N,p},\qquad
\sigma_e=\prod_{i\in e}\sigma_i\quad(\sigma\in\Sigma_N,\ e\in\mathcal E_N).
\end{equation*}
The disorder variables~$J=(J_e)_{e\in\mathcal E_N}$ has independent entries with common law~$\mu$. Write~$\E$ and~$\pr$ for expectation and probability over the disorder. For~$J_0$ with law~$\mu$, assume
\begin{equation}\label{eq:disorder}
J_0\overset{\mathrm d}=-J_0,\qquad \E J_0^2=1\,.
\end{equation}
Namely the disorder is symmetric about zero and has unit variance.

For inverse temperature~$\beta\geq0$ and a homogeneous external field~$h\in\mathbb R$, define the Hamiltonian and partition function
\begin{align}
H_N(\sigma;J)&=\sum_{p\in\mathcal P}\frac{a_p}{N^{(p-1)/2}}\sum_{e\in\mathcal E_{N,p}}J_e\sigma_e,\label{eq:hamiltonian}\\
Z_N(\beta,h;J)&=\sum_{\sigma\in\Sigma_N}\exp\!\left\{\beta H_N(\sigma;J)+h\sum_{i=1}^N\sigma_i\right\}\,.\label{eq:partition}
\end{align}
These interaction conventions are consistent with~\cite{DW23} for SK and~\cite{DW26}*{(1)--(2)} for higher-order mixtures. Following~\cite{DW26}*{(6)}, define the entropy function
\begin{equation}\label{eq:sign-entropy}
I(x)=\frac{1+x}{2}\log(1+x)+\frac{1-x}{2}\log(1-x),\qquad -1\leq x\leq1\,,
\end{equation}
 and the structure function and second-moment threshold are respectively
\begin{equation}\label{eq:graphical-threshold}
\xi(x)=\sum_{p\in\mathcal P}\frac{a_p^2x^p}{p!},\qquad
\beta_{\mathrm f}^{\,2}=\inf_{0<x\leq1}\frac{I(x)}{\xi(x)}\,.
\end{equation}
For~$0<x\leq1$, we have~$\xi(x)\leq\xi(1)x^2$ and~$I(x)\geq x^2/2$. The latter follows from~$I(0)=I'(0)=0$ and~$I''(x)=1/(1-x^2)\geq1$. Evaluating the ratio also at~$x=1$ gives
\begin{equation*}
0<\frac{1}{2\xi(1)}\leq\beta_{\mathrm f}^{\,2}\leq\frac{\log2}{\xi(1)}<\infty.
\end{equation*}
For~$\beta<\beta_{\mathrm f}$, we have~$\beta^2\xi(x)<I(x)$ for every~$x\in(0,1]$. This entropy comparison plays a key role in the second moment estimates in Section~\ref{sec:cluster-expansion}. For SK, $\xi(x)=x^2/2$ and~$I(x)/(x^2/2)\to1$ as~$x\downarrow0$, it gives~$\beta_{\mathrm f}=1$, which is the static phase transition threshold for SK model.

Fix~$0<\beta_0<\beta_{\mathrm f}$ and~$\rho_0>0$. We work in the following regime
\begin{equation}\label{eq:field-range}
0\leq\beta\leq\beta_0,\qquad |h|\leq\rho_0N^{-1/4}\,.
\end{equation}
The external field may depend on~$N$, and~$h=0$ is included. Constants may depend on the fixed parameters~$\mathcal P,(a_p),\beta_0,\rho_0$, but not on~$N$ or the particular~$\beta,h$ in this range.

\subsection{Approximate counting problem}\label{subsec:problem}

The algorithmic input consists of~$N$, a disorder realization~$J$, the parameters~$\beta,h$, a relative accuracy~$\eta\in(0,1)$, and a failure allowance~$\delta\in(0,1)$. The goal is to return~$\widetilde Z_N$ satisfying
\begin{equation}\label{eq:relative-objective}
\left|\frac{\widetilde Z_N}{Z_N(\beta,h;J)}-1\right|\leq\eta
\end{equation}
with a failure probability bound of~$\delta$ plus the disorder-dependent term in the main theorem. Accuracy is assessed over the disorder and, for SK, independent algorithmic randomness. We include this additional randomness in~$\E$ and~$\pr$ whenever an algorithmic output is involved. 

To obtain relative accuracy~$\eta$, it suffices to obtain $\hat{L}$ such that $|\widehat L_N-\log Z_N|\leq\varepsilon$ for $\varepsilon=\log(1+\eta)$. The estimate~$\widetilde Z_N=\exp(\widehat L_N)$ satisfies
\begin{equation}\label{eq:log-to-relative}
\frac{1}{1+\eta}\leq\frac{\widetilde Z_N}{Z_N}\leq1+\eta\,,
\end{equation}
and hence~\eqref{eq:relative-objective}. This needs a finer control than an~$o(1)$ approximation to the free energy~$N^{-1}\log Z_N$, which determines~$Z_N$ only up to a factor~$e^{o(N)}$.

Here the computational cost is measured by counting the number of exact operations, such as real arithmetic and evaluations of elementary functions. Our bounds do not account for bit complexity or the numerical precision needed to control roundoff errors.

\subsection{Main result}\label{subsec:main-results}

The following main result gives the guarantee for the algorithm~\textsc{LogApprox}, stated in Section~\ref{subsec:algorithms}. 

\begin{theorem}\label{thm:finite-size}
There are constants~$C,c>0$, depending only on~$\mathcal P,(a_p),\beta_0,\rho_0$, such that Algorithm~\ref{alg:log-approx}, for every~$N\geq P$, specified~$(\beta,h)$ satisfying~\eqref{eq:field-range}, and~$\varepsilon,\delta\in(0,1)$, returns~$\widehat L_N$ in at most
\begin{equation}\label{eq:finite-size-cost}
CN^C\varepsilon^{-C}\delta^{-C}
\end{equation}
operations for every finite real input~$J$ and every realization of the internal randomness. For~$J_0$ with the common disorder law~$\mu$, set
\begin{align}
\tau_N=\sum_{p\in\mathcal P}\binom Np\pr\bigl(|a_pJ_0|>cN^{(p-1)/2}\bigr)\,.\label{eq:disorder-tail}
\end{align}
Under~\eqref{eq:disorder}, the error satisfies
\begin{align}
\pr\bigl(|\widehat L_N-\log Z_N(\beta,h;J)|>\varepsilon\bigr)&\leq\delta+\tau_N.\label{eq:finite-size-error}
\end{align}
For mixtures with~$p_m\geq3$, the algorithm is deterministic and probability is over~$J$. For SK, the algorithm is randomized and probability also includes its independent internal randomness. 
\end{theorem}

 The term $\tau_N$ is due to a union bound to control the probability of some~$|a_pJ_e|$ exceeding~$cN^{(p-1)/2}$.  For any fixed parameter values within the stated ranges, the failure probability is at most~$\delta+\tau_N$. The event of success may depend on that choice, so we do not claim simultaneous accuracy over the entire parameter range. The cost exponent grows as~$\beta_0$ approaches~$\beta_{\mathrm f}$.

Note that a second moment condition on disorder alone need not make the bound on the tail probability small. The high probability consequences require higher moments control, such as  a finite fourth moment in the following Remark~\ref{rem:moment-consequences}.

\begin{remark}[Consequences under higher moments conditions]\label{rem:moment-consequences}\label{thm:relative-counting}\label{cor:polynomial-confidence}
Under~\eqref{eq:disorder}, bounds on~$\tau_N$ give the following consequences.
\begin{enumeratei}
\item If~$\E|J_0|^4<\infty$, then~$t^4\pr(|J_0|>t)\to0$ as~$t\to\infty$. Substituting the thresholds in~\eqref{eq:disorder-tail} gives~$\tau_N=\sum_{p\in\mathcal P}o(N^{2-p})=o(1)$, including SK. For relative tolerance~$\eta\in(0,1)$, apply the theorem with logarithmic tolerance~$\varepsilon=\log(1+\eta)$ and~$\delta=1/(2N)$. Then~$\widetilde Z_N=\exp(\widehat L_N)>0$ satisfies
\begin{equation}\label{eq:relative-counting-result}
\pr\left(\frac{\widetilde Z_N}{Z_N(\beta,h;J)}\in[1-\eta,1+\eta]\right)
\geq1-\tau_N-\frac1{2N}=1-o_N(1)\,.
\end{equation}
Since~$\varepsilon\geq\eta/2$, the cost is at most~$CN^C\eta^{-C}$ after enlarging~$C$. 
\item Fix~$b>0$. If~$\E|J_0|^m<\infty$ for some~$m\geq2(p_m+b)/(p_m-1)$, then
\begin{equation*}
t^m\pr(|J_0|>t)\leq\E\bigl[|J_0|^m\ind_{\{|J_0|>t\}}\bigr]\longrightarrow0
\qquad(t\longrightarrow\infty).
\end{equation*}
Substituting the thresholds in~\eqref{eq:disorder-tail}, since~$m>2$ and every exponent is at most~$-b$, it gives~$\tau_N=\sum_{p\in\mathcal P}o(N^{p-m(p-1)/2})=o(N^{-b})$. Taking~$\delta=N^{-b}/2$ gives logarithmic error at most~$\varepsilon$ with failure probability at most~$N^{-b}$ for all sufficiently large~$N$, at cost~$C_bN^{C_b}\varepsilon^{-C_b}$. The constant~$C_b$ depends only on the fixed model, $\beta_0,\rho_0,b$. 
\end{enumeratei}
\end{remark}

\subsection{The algorithm}\label{subsec:algorithms}

We first separate an explicitly computable factor from the partition function $Z_N$, which is basically the first step of cluster expansion in~\cite{ALR87,DW23,DW26}. For each interaction~$e$, let
\begin{equation*}
x_e=\frac{\beta a_{|e|}J_e}{N^{(|e|-1)/2}}.
\end{equation*}
Define
\begin{equation}\label{eq:algorithm-normalization}
\overline Z_N=\prod_{e\in\mathcal E_N}\cosh x_e,
\qquad \widehat Z_N=\frac{Z_N}{(2\cosh h)^N\,\overline Z_N}\,.
\end{equation}
Here~$\widehat Z_N$ is a normalized partition function. The algorithm thus computes~$N\log(2\cosh h)+\log\overline Z_N$ exactly and approximates~$\log\widehat Z_N$.

For a target mean-square error~$u\in(0,1)$ in approximating the normalized partition function~$\widehat Z_N$, the subroutine~\textsc{GraphApprox} (Algorithm~\ref{alg:graph-approx}) returns an approximation~$T$ satisfying
\(
\E|T-\widehat Z_N|^2\leq u
\)
in~$CN^Cu^{-C}$ operations (Proposition~\ref{prop:mean-square}). To obtain a logarithmic approximation, we also need to know that~$\widehat Z_N$ is unlikely to be very small. Proposition~\ref{prop:lower-tail} provides fixed constants~$C_+\geq1$ and~$s>0$ such that, for every ~$z>0$,
\begin{equation*}
\pr(\widehat Z_N<z)\leq C_+z^s+\tau_N.
\end{equation*}
Their explicit values are given in~\eqref{eq:algorithm-mass-constant} and~\eqref{eq:lower-tail-constants}.

\begin{algorithm}[!htbp]
\caption{\textsc{LogApprox}$(N,\beta,h,J,\varepsilon,\delta)$}\label{alg:log-approx}
\begin{algorithmic}[1]
\Require Disorder~$J$ with finite real entries, $N\geq P$, parameters~$(\beta,h)$ as in~\eqref{eq:field-range}, additive logarithmic tolerance~$\varepsilon\in(0,1)$, and failure allowance~$\delta\in(0,1)$.
\State $L_{\mathrm{base}}\gets N\log(2\cosh h)+\sum_{e\in\mathcal E_N}\log\cosh x_e$.
\State $z\gets(\delta/(2C_+))^{1/s}$ and~$u\gets\delta\varepsilon^2z^2/8$.
\State $T\gets\textsc{GraphApprox}(N,\beta,h,J,u)$.
\If{$T>0$}
\State \Return $\widehat L_N=L_{\mathrm{base}}+\log T$.
\Else
\State \Return $\widehat L_N=L_{\mathrm{base}}$.
\EndIf
\end{algorithmic}
\end{algorithm}

For these choices of $z$ and $u$, the lower-tail bound and Markov’s inequality imply that $\widehat Z_N\geq z$ and $|T-\widehat Z_N|\leq\varepsilon z/2$ hold together with probability at least $1-\delta-\tau_N$. The error is then small relative to $\widehat Z_N$, consequently $|\log T-\log\widehat Z_N|\leq\varepsilon$. Since $u^{-1}$ is bounded by a polynomial in $\varepsilon^{-1}$ and $\delta^{-1}$, the computational cost is polynomial jointly in $N,\varepsilon^{-1},\delta^{-1}$.


\subsection{Ideas of the proof}\label{subsec:proof-sketch} The proof for Theorem~\ref{thm:finite-size} contains three major steps. The first step is using the cluster expansion in~\cite{ALR87,DW23,DW26} to obtain estimates for the variance order of large (hyper)-graph clusters, which determines what graphs to be discarded. The second step is an efficient algorithmic procedure to sum the retained graphs. The last step is about controlling the lower tail probability of $\hat{Z}_N$, which translates the approximation error into logarithmic accuracy.

\paragraph{Expansion and error estimate.} Let~$\omega_e=\tanh x_e$ and~$\hh=\tanh h$. For a graph~$\Gamma\subseteq\mathcal E_N$, write the degree of vertex $i$ in $\Gamma$ as~$\deg_\Gamma(i)=|\{e\in\Gamma:i\in e\}|$ and let~$\partial\Gamma=\{i\in[N]:\deg_\Gamma(i)\text{ is odd}\}$ be the odd degree vertices. One can check that 
\begin{equation*}
\widehat Z_N=\E_h\prod_{e\in\mathcal E_N}(1+\omega_e\sigma_e) =\sum_{\Gamma\subseteq\mathcal E_N}W_h(\Gamma),
\qquad \text{with} \quad W_h(\Gamma)=\hh^{|\partial\Gamma|}\prod_{e\in\Gamma}\omega_e,
\end{equation*}
where $\E_h\sigma_i=\hh$.  Independence and symmetry of the disorder make the cross terms vanish when we square the error and take expectation. Thus, for a fixed family~$\mathcal R$ of graphs,
\begin{equation*}
\E\left|\widehat Z_N-\sum_{\Gamma\in\mathcal R}W_h(\Gamma)\right|^2
=\sum_{\Gamma\notin\mathcal R}\E W_h(\Gamma)^2.
\end{equation*}
We later choose which graphs to keep so that the sum on the right is small. In the mean time, we also need to compute the total weight of retained sum in polynomial time.

\paragraph{Mixed $p$-spin analysis} For an interaction of order~$p\ge 3$, $\E\omega_e^2=O(N^{-(p-1)})$. Lemma~\ref{lem:higher-support-mass}, based on similar combinatorial analysis in~\cite{DW26}, bounds the sum of second moments over all hypergraphs using~$v$ vertices by~$(Cv/N)^{\kappa v}$ for fixed~$C,\kappa>0$. This bound decreases rapidly when~$v$ is a small fraction of~$N$. We keep all hypergraphs using at most~$k$ (to be chosen) vertices. The next step is how to efficiently compute the retained sum. For a fixed set~$S$ with~$|S|\leq k$, let~$G(S)$ be the average of~$\prod_{e\in\mathcal E_N,\,e\subseteq S}(1+\omega_e\sigma_e)$ over independent spins on~$S$ with mean~$\hh$. It sums the weights of all hypergraphs contained in~$S$. Some hypergraph might be counted more than once in different averages. We apply an inclusion-exclusion trick to avoid over-counting. For each set $S$, $G(S)$ will be evaluated once using~$2^{|S|}$ spin configurations. In total, there are~$\sum_{j=0}^{k}\binom Nj2^j$ spin configurations, with polynomial work per configuration. 

\paragraph{Analysis for SK} In the SK case, for fixed cycle length~$\ell$, the number of cycles is of order~$N^\ell$, and each has second moment of order~$N^{-\ell}$. The sum of their second moments therefore has constant order. We instead truncate by \emph{branching vertices}, i.e. those of degree at least three. For~$b\geq1$, Lemma~\ref{lem:sk-branching-mass} bounds the sum of second moments for graphs with~$b$ branching vertices by~$C(Cb/N)^b$. We keep at most~$k$ branching vertices and~$L$ edges to exclude long paths and cycles. Fix the branching set and add an auxiliary vertex and join it to each odd-degree vertex by an edge of weight~$\hh$ to account for the external field effects. Every used vertex outside this enlarged set now has degree two. Apart from edges within the enlarged set, the resulting graph consists of paths and cycles whose outside vertices do not overlap. Color coding~\cite{AYZ95} provides an efficient way to compute the sum of those graph structures. The cost per branching set and coloring is polynomial in~$N$ and exponential in~$k+L$ (Section~\ref{sec:sk}).

\paragraph{Completing the proof.} We finally choose a proper cutoff to obtain a mean-square error at most~$u$ at a cost polynomial in~$N$ and~$u^{-1}$. Here~$k$ counts all used vertices for higher-order mixtures and branching vertices for SK. For cutoffs~$k$ below a small fixed fraction of~$N$, the mean-square truncation error bounds decrease exponentially in~$k\log(N/k)$, while the number of vertex sets of size~$k$ is at most~$(eN/k)^k$. Balancing these two bounds gives the required cutoffs. For SK, the contribution from graphs with more than~$L$ edges decreases exponentially in~$L$ below the second moment threshold $\beta_f$. Choosing~$L$ large enough to make this contribution at most a fixed fraction of~$u$ requires only~$L=O(1+\log(1/u))$. With~$R$ independent colorings for each candidate branching set, the mean-square error in estimating the retained graph sum is at most~$C_+e^{2L}/R$. Taking~$R\geq8C_+e^{2L}/u$ makes this error at most~$u/8$. Therefore, the required number of colorings is also polynomial in~$u^{-1}$.

To turn this absolute approximation into logarithmic accuracy, we must ensure that~$\widehat Z_N$ is unlikely to be very small. We obtain the required lower-tail bound via reverse hypercontractivity~\cite{MOS13}. Combining this bound with the mean-square approximation gives the desired logarithmic accuracy with failure probability at most~$\delta+\tau_N$, while preserving polynomial cost.

\subsection{Related work}\label{subsec:related-work}


Barvinok's interpolation method approximates the partition function of Ising models by computing a truncated Taylor series. Bounds on the locations of the complex zeros control the truncation error, which provides the basis for the zero-freeness approach to approximate counting. For applications to classical spin systems and related counting problems, see~\cite{Bar17,BB21,PR17,LSS19}. In the spin glass setting, Bencs, Huang, Lee, Liu and Regts~\cite{BHL25} first developed the zero-freeness approach and obtain deterministic algorithms with quasi-polynomial running time for finite Gaussian mixed Ising and spherical models at zero external field, with high probability over the disorder at a fixed margin below the second-moment threshold. Their models include mixtures of pair and higher-order interactions, indexed by ordered tuples with repeated sites allowed. After matching interaction normalizations, their Ising threshold agrees with~$\beta_{\mathrm f}$.

Sampling also provides a route to approximate counting by estimating successive ratios of partition functions along a temperature schedule. For SK, functional inequalities yield efficient sampling at sufficiently high temperature, as shown by Eldan, Koehler, and Zeitouni~\cite{EKZ22} and Anari, Koehler, and Vuong~\cite{AKV24}. For Gaussian SK at zero external field, the stochastic localization approach of El Alaoui, Montanari, and Sellke, together with the work of Celentano, gives efficient sampling in normalized Wasserstein distance throughout~$\beta<1$~\cites{AMS25,Cel24}, see some recent progress on the sampling problem in~\cite{LSS26,Wang26,BLO26}.

To obtain accurate counts by annealing, one must control the sampling error at every temperature used in the schedule. Hou, Liu, and Yu~\cite{HLY26}*{Theorem~1.2 and Corollary~1.3} provide such control through a weak Poincar\'e inequality and simulated annealing with Glauber dynamics for Gaussian SK at zero external field throughout~$\beta<1$. Their counting algorithm achieves prescribed inverse polynomial relative accuracy and failure probabilities in polynomial time for sufficiently large~$N$. The polynomial exponent may depend on the prescribed accuracy and confidence powers.

The spectral-gap argument in the OpenAI mathematics release~\cite{OAI26gap} provides another sampling-based route for Gaussian SK at zero external field. As noted by Harbuzova--Lepsveridze--Majid--Moitra~\cite{HLMM26}*{Section~1.2}, making the estimates uniform on fixed subcritical temperature intervals and applying annealing yields randomized polynomial-time relative approximation at every fixed~$\beta<1$, with high probability over the disorder. These approaches obtain counts from Gibbs samples, whereas our algorithm directly sums retained graph weights. Conversely, the counting-to-sampling reduction of Jerrum--Valiant--Vazirani~\cite{JVV86} requires counts for models obtained by fixing spins. Such conditioning can introduce site-dependent external fields and lower-order interactions outside our model class, so our theorem does not directly provide the required counting oracle. 

For related algorithmic results in other disordered systems, see~\cites{HLP23,AEGP26,Han26} for the random field Ising model and~\cites{Ho25,LW25} for the continuous random energy model.


\subsection{Concurrent work}\label{subsec:concurrent-work}
Few weeks before this work~\footnote{A preliminary draft of this work and a companion Lean formalization were posted on \href{https://github.com/qiangwu2/AI-assisted-results/tree/main/approximate-counting-spin-glasses}{GitHub} on October~7, after OpenAI's mathematical results release.}, Harbuzova et~al.~\cite{HLMM26}*{Theorem~1.1} and Schmidhuber and Zlokapa~\cite{SZ26}*{Theorems~26 and~31} obtained similar results. In particular, they have deterministic algorithms for approximate counting in Gaussian SK at zero external field throughout~$\beta<1$, with polynomial running time and high probability of success over the disorder. Our SK algorithm is randomized and allows symmetric disorder and homogeneous weak external fields, subject to an explicit disorder tail bound. For mixed $p$-spin case, \cite{HLMM26} also covers models containing both pair interactions and interactions of higher order, which our results exclude. After matching interaction normalizations, their threshold agrees with~$\beta_{\mathrm f}$. For mixtures whose orders are all at least three, \cite{SZ26} reaches the static phase transition threshold, which can lie above~$\beta_{\mathrm f}$.

The complexity guarantees also differ. Similar to~\cite{HLMM26}, we obtain a bound on the number of exact operations that is polynomial jointly in system size, inverse requested accuracy, and inverse failure allowance. Our failure bound includes the disorder tail term~$\tau_N$, while theirs includes a vanishing term from a zero-freeness estimate. The guarantees in~\cite{SZ26} allow accuracy and failure targets given by any fixed inverse powers of~$N$, with polynomial running time whose exponent may depend on those powers. They also establish polynomial bit complexity, whereas our operation bounds do not account for numerical precision.

For SK at zero external field, all three methods separate branching vertices from paths and cycles through vertices of degree two and use color coding to sum these pieces. The construction is particularly close to~\cite{HLMM26}*{Section~5}, which also combines piece weights through polynomial multiplication. Both concurrent algorithms are deterministic. We average independent random colorings and introduce an auxiliary vertex to account for the weak external field effects. For mixtures whose orders are all at least three, our algorithm sums the retained hypergraphs exactly through local spin averages and inclusion--exclusion. The corresponding methods in~\cite{HLMM26}*{Appendix~A} and~\cite{SZ26}*{Section~7} use an extension of color coding and a Hermite expansion with cavity calculations, respectively. Our treatment of small partition function values also differs from~\cite{HLMM26}. We use reverse hypercontractivity for the coupling signs, rather than zero-freeness, to convert approximation in mean square into logarithmic accuracy.

\subsection{AI disclosure}\label{subsec:ai-disclosure}

This work was developed through extensive interactions with GPT-5.6 Pro and GPT-6 Pro. The author originally proposed using the cluster expansion to obtain approximate counting algorithms by discarding large graph clusters, motivated by the work of Dey--Wu~\cites{DW23,DW26}. In particular, the fast decay of the second moments of large hypergraph clusters in the higher-order models suggested that a relatively small retained sum might suffice. Starting from this idea, GPT-5.6 Pro proposed a polynomial time algorithm for the higher order case based on a cutoff on the number of hyperedges. However, the initial approach only gave a quasi-polynomial time algorithm for SK.

Further interactions with  GPT-6 Pro focused on improving the SK case. The models suggested using the established color-coding method~\cite{AYZ95} and contributed to the development of the truncation schemes and key arguments in Sections~\ref{sec:sk} and \ref{sec:error-complexity}, including the application of reverse hypercontractivity to lower-tail control. 
We also used Codex to assist with literature searches, figure generation, and the development of a companion Lean formalization (\href{https://github.com/qiangwu2/approximate-counting-spin-glasses}{GitHub}). The author checked all proofs, revised the manuscript, and takes responsibility for its mathematical claims.

\section{Cluster expansion and second moment estimates}\label{sec:cluster-expansion}

In this section, we establish second moment bounds for $\hat{Z}_N$ shared by the algorithms. Orthogonality of distinct graph clusters expresses the mean-square truncation error as a sum of expected squared graph weights. A bound at an inverse temperature below the threshold $\beta_f$ controls this sum and gives exponential decay in the number of edges. Throughout Sections~\ref{sec:cluster-expansion}--\ref{sec:error-complexity}, the assumptions of Section~\ref{subsec:model} hold, in particular~$N\geq P$, $0\leq\beta\leq\beta_0<\beta_{\mathrm f}$, and~$|h|\leq\rho_0N^{-1/4}$.

\subsection{Expansion and orthogonality}\label{subsec:expansion}

Recall~$\omega_e=\tanh x_e$ and~$\hh=\tanh h$. An interaction set~$\Gamma\subseteq\mathcal E_N$ is a graph for SK and a hypergraph otherwise. Let~$V(\Gamma)=\bigcup_{e\in\Gamma}e$ be the vertex set of $\Gamma$ and $\partial\Gamma$ be its odd-degree vertices. Define the graph weights by
\begin{equation*}
\omega(\Gamma)=\prod_{e\in\Gamma}\omega_e,
\qquad W_h(\Gamma)=\hh^{|\partial\Gamma|}\omega(\Gamma).
\end{equation*}
Unless restrictions are stated, sums over~$\Gamma$ range over all subsets of~$\mathcal E_N$, including the empty graph. Empty products and zeroth powers equal one.

For any external field~$h\in\mathbb R$, let~$\pr_h$ be the law of independent spins with
\begin{equation*}
\pr_h(\sigma_i=\pm1)=\frac{e^{\pm h}}{2\cosh h},\qquad i\in[N],
\end{equation*}
and let~$\E_h$ be its expectation. In particular, $\E_h\sigma_i=\hh$.  Applying~$e^{x_e\sigma_e}=\cosh(x_e)(1+\omega_e\sigma_e)$ to each interaction in~$Z_N$, then using the normalization~\eqref{eq:algorithm-normalization}, gives~$\widehat Z_N=\E_h\prod_e(1+\omega_e\sigma_e)$. In the product expansion, selecting the interactions in~$\Gamma$ contributes~$\omega(\Gamma)\prod_{e\in\Gamma}\sigma_e$. Independence of the spins gives
\begin{equation*}
\E_h\prod_{e\in\Gamma}\sigma_e
=\prod_{i\in V(\Gamma)}\E_h\sigma_i^{\deg_\Gamma(i)}
=\hh^{|\partial\Gamma|}.
\end{equation*}
The last equality uses~$\sigma_i^2=1$ and~$\E_h\sigma_i=\hh$. Summing over~$\Gamma$ gives the expansion of~\cite{DW23} and~\cite{DW26}*{(3)--(5), (38), and~(40)}
\begin{equation}\label{eq:algorithm-expansion}
\widehat Z_N=\E_h\prod_{e\in\mathcal E_N}(1+\omega_e\sigma_e)
=\sum_{\Gamma\subseteq\mathcal E_N}W_h(\Gamma)>0\,.
\end{equation}
Positivity follows from~$|\omega_e|<1$. Symmetry gives~$\E\omega_e=0$. For~$e\in\mathcal E_{N,p}$, denote the common second moment~$\E\omega_e^2$ by~$r_{N,p}$. Since~$\tanh^2x\leq x^2$ and~$\E J_0^2=1$,
\begin{equation}\label{eq:edge-second-moment}
r_{N,p}:=\E\tanh^2\!\left(\frac{\beta a_pJ_0}{N^{(p-1)/2}}\right)
\leq\frac{\beta_0^2a_p^2}{N^{p-1}}\,.
\end{equation}
We allow random coefficients in the next identity because the SK evaluator uses colors drawn independently of the disorder.

\begin{lemma}\label{lem:orthogonality}
For~$\Gamma,\Gamma'\subseteq\mathcal E_N$,
\begin{equation}\label{eq:orthogonality}
\E[\omega(\Gamma)\omega(\Gamma')]=\ind_{\{\Gamma=\Gamma'\}}\prod_{e\in\Gamma}r_{N,|e|}\,.
\end{equation}
Let~$(c_\Gamma)_{\Gamma\subseteq\mathcal E_N}$ be square-integrable real random coefficients whose full array is independent of~$J$. With expectation over both~$J$ and this array,
\begin{equation}\label{eq:random-orthogonality}
\E\left|\sum_\Gamma c_\Gamma W_h(\Gamma)\right|^2
=\sum_\Gamma\E[c_\Gamma^2]\hh^{2|\partial\Gamma|}\prod_{e\in\Gamma}r_{N,|e|}\,.
\end{equation}
The coefficients may depend on one another.
\end{lemma}
\begin{proof}
If~$\Gamma\ne\Gamma'$, at least an edge in their symmetric difference appears once in the product. Its centered weight is independent of the other edge weights, so the expectation vanishes. If~$\Gamma=\Gamma'$, independence gives the second moments. Conditioning on the coefficient array and expanding the square proves~\eqref{eq:random-orthogonality}.
\end{proof}

In particular, for a deterministic family~$\mathcal R$ of retained graphs,
\begin{equation}\label{eq:truncation-isometry}
\E\left|\widehat Z_N-\sum_{\Gamma\in\mathcal R}W_h(\Gamma)\right|^2
=\sum_{\substack{\Gamma\subseteq\mathcal E_N\\\Gamma\notin\mathcal R}}
\hh^{2|\partial\Gamma|}\prod_{e\in\Gamma}r_{N,|e|}\,.
\end{equation}
Thus the error is determined by the expected squared weights of the discarded graphs. We next give the second moment estimates for those graphs.

\subsection{Second moments and the edge tail}\label{subsec:expansion-mass}

Choose a comparison inverse temperature strictly between~$\beta_0$ and~$\beta_{\mathrm f}$ by setting
\begin{equation}\label{eq:algorithm-temperature-constants}
\beta_+=\frac{\beta_0+\beta_{\mathrm f}}2,\qquad
\theta=\frac{\beta_0}{\beta_+}\in(0,1)\,.
\end{equation}
For~$p\in\mathcal P$, define the comparison weights by
\begin{equation*}
\bar r_{N,p}=\frac{\beta_+^2a_p^2}{N^{p-1}},
\qquad r_{N,p}\leq\theta^2\bar r_{N,p}.
\end{equation*}
Here~$\bar r_{N,p}$ is a deterministic upper bound, not the exact edge second moment at~$\beta_+$. Bounding the graph sum with these larger weights will leave a factor~$\theta^2$ for each original edge. Let
\begin{equation*}
\Delta=\sum_{p\in\mathcal P}\frac{a_p^2}{2(p-2)!},
\qquad \lambda_+=\frac{\beta_+^2}{\beta_{\mathrm f}^2}<1.
\end{equation*}
Define
\begin{equation}\label{eq:algorithm-mass-constant}
C_+=\exp\!\left\{\beta_+^2\Delta+\frac{1+\lambda_+}{2(1-\lambda_+)}\rho_0^4\right\}
\frac{3+\lambda_+}{1-\lambda_+}\,.
\end{equation}
The proof follows the second-moment arguments of~\cite{DW23}*{Lemma~4.1} and~\cite{DW26}*{Lemma~3.1 and Proposition~3.3}, with an explicit bound for every~$N\geq P$ under~\eqref{eq:disorder}.

\begin{proposition}[Second moments and edge tail]\label{prop:mass}
For every~$N\geq P$ and~$|h|\leq\rho_0N^{-1/4}$,
\begin{equation}\label{eq:total-mass}
\sum_{\Gamma\subseteq\mathcal E_N}\hh^{2|\partial\Gamma|}
\prod_{e\in\Gamma}\bar r_{N,|e|}\leq C_+\,.
\end{equation}
For every real~$t\geq0$,
\begin{equation}\label{eq:edge-tail}
\sum_{\substack{\Gamma\subseteq\mathcal E_N\\|\Gamma|>t}}
\hh^{2|\partial\Gamma|}\prod_{e\in\Gamma}r_{N,|e|}
\leq C_+\theta^{2(\lfloor t\rfloor+1)}\leq C_+\theta^{2t}\,.
\end{equation}
\end{proposition}
\begin{proof}
To prove~\eqref{eq:total-mass}, we rewrite the graph sum as a spin average. Its external field factor is~$\hh^{2|\partial\Gamma|}$, so the new spins should have mean~$\hh^2$. Define the tilted external field by
\begin{equation*}
\widetilde h=\operatorname{arctanh}(\hh^2)=\frac12\log\cosh(2h).
\end{equation*}
Then~$\tanh\widetilde h=\hh^2$ and~$0\leq\widetilde h\leq h^2$, by~$\log\cosh y\leq y^2/2$. Applying the product expansion with external field~$\widetilde h$ and edge weights~$\bar r_{N,|e|}$ gives
\begin{equation*}
\sum_\Gamma\hh^{2|\partial\Gamma|}\prod_{e\in\Gamma}\bar r_{N,|e|}
=\E_{\widetilde h}\prod_{e\in\mathcal E_N}(1+\bar r_{N,|e|}\sigma_e).
\end{equation*}
To use the entropy comparison, we express the interaction sum in terms of~$S_N=\sum_{i=1}^N\sigma_i$. For~$p\in\mathcal P$, let~$R_{N,p}$ be the sum of terms in~$S_N^p$ indexed by ordered tuples with a repeated site. There are at most~$\binom p2N^{p-1}$ such tuples, each contributing a term of absolute value one. Thus
\begin{equation*}
S_N^p=p!\sum_{e\in\mathcal E_{N,p}}\sigma_e+R_{N,p},
\qquad |R_{N,p}|\leq\binom p2N^{p-1}.
\end{equation*}
For each~$p\in\mathcal P$, divide by~$p!$ and use~$\binom p2/p!=1/(2(p-2)!)$. Multiplying by~$\bar r_{N,p}=\beta_+^2a_p^2/N^{p-1}$ gives
\begin{equation*}
\bar r_{N,p}\sum_{e\in\mathcal E_{N,p}}\sigma_e
\leq \frac{\beta_+^2a_p^2}{p!}\,N\left(\frac{S_N}{N}\right)^p
+\frac{\beta_+^2a_p^2}{2(p-2)!}.
\end{equation*}
Summing over~$p$, the first terms give~$\beta_+^2N\xi(S_N/N)$ and the remainder terms give~$\beta_+^2\Delta$. Thus
\begin{equation*}
\sum_{e\in\mathcal E_N}\bar r_{N,|e|}\sigma_e
\leq\beta_+^2N\xi(S_N/N)+\beta_+^2\Delta.
\end{equation*}

All coefficients of~$\xi$ are nonnegative, so~$\xi(x)\leq\xi(|x|)$. The definition of~$\beta_{\mathrm f}$ and evenness of~$I$ therefore imply~$\beta_+^2\xi(x)\leq\lambda_+I(x)$ on~$[-1,1]$. To bound the product by the exponential of this sum, we also need its factors to be nonnegative. The inequality~$\beta_+^2\xi(1)<I(1)=\log2$ gives~$\beta_+^2a_p^2<p!\log2$. Since~$p!\leq p^{p-1}\leq N^{p-1}$, we have~$\bar r_{N,p}<1$. Define~$X_N=NI(S_N/N)$. For each spin configuration, all factors~$1+\bar r_{N,|e|}\sigma_e$ are positive. Applying~$1+x\leq e^x$ to each factor gives
\begin{equation*}
\prod_{e\in\mathcal E_N}(1+\bar r_{N,|e|}\sigma_e)
\leq\exp\!\left(\sum_{e\in\mathcal E_N}\bar r_{N,|e|}\sigma_e\right)
\leq e^{\beta_+^2\Delta+\lambda_+X_N}.
\end{equation*}
Taking~$\E_{\widetilde h}$ bounds the graph sum in~\eqref{eq:total-mass} by~$e^{\beta_+^2\Delta}\E_{\widetilde h}e^{\lambda_+X_N}$.

It remains to bound~$\E_{\widetilde h}e^{\lambda_+X_N}$ uniformly in~$N$. First change to the uniform spin measure~$\pr_0$. The probability of a configuration under~$\pr_{\widetilde h}$ is~$2^{-N}e^{\widetilde hS_N}/\cosh(\widetilde h)^N$, so
\begin{equation*}
\E_{\widetilde h}e^{\lambda_+X_N}
=\frac{\E_0\!\left[e^{\lambda_+X_N}e^{\widetilde hS_N}\right]}
{\cosh(\widetilde h)^N}.
\end{equation*}
We apply H\"older's inequality under~$\pr_0$ to the numerator. Set
\begin{equation*}
q_1=\frac{1+\lambda_+}{2\lambda_+},
\qquad q_2=\frac{1+\lambda_+}{1-\lambda_+},
\qquad v=q_1\lambda_+=\frac{1+\lambda_+}{2}<1.
\end{equation*}
Since~$q_1,q_2>1$ and~$1/q_1+1/q_2=1$, H\"older's inequality gives
\begin{equation*}
\E_0\!\left[e^{\lambda_+X_N}e^{\widetilde hS_N}\right]
\leq\bigl(\E_0e^{vX_N}\bigr)^{1/q_1}
\bigl(\E_0e^{q_2\widetilde hS_N}\bigr)^{1/q_2}.
\end{equation*}
Under~$\pr_0$, the spins are independent and uniform on~$\{-1,1\}$, so
\begin{equation*}
\E_0e^{q_2\widetilde hS_N}
=\prod_{i=1}^N\E_0e^{q_2\widetilde h\sigma_i}
=\left(\frac{e^{q_2\widetilde h}+e^{-q_2\widetilde h}}2\right)^N
=\cosh(q_2\widetilde h)^N.
\end{equation*}
Substituting this identity and dividing by~$\cosh(\widetilde h)^N$ yields
\begin{equation*}
\E_{\widetilde h}e^{\lambda_+X_N}
\leq\bigl(\E_0e^{vX_N}\bigr)^{1/q_1}
\frac{\cosh(q_2\widetilde h)^{N/q_2}}{\cosh(\widetilde h)^N}.
\end{equation*}
To bound~$\E_0e^{vX_N}$, we first estimate the tail probability of~$X_N$. Since~$I$ is even, $X_N=NI(|S_N|/N)$. Strict monotonicity of~$I$ on~$[0,1]$ gives, for~$0\leq x\leq1$,
\begin{align*}
\pr_0(X_N>NI(x))
=\pr_0(|S_N|>Nx)
\leq2\pr_0(S_N\geq Nx)
\leq2e^{-NI(x)}.
\end{align*}
The two inequalities use symmetry of~$S_N$ and the Chernoff bound, respectively. As~$x$ ranges over~$[0,1]$, $NI(x)$ varies continuously from~$0$ to~$N\log2$. Writing~$t=NI(x)$ therefore gives~$\pr_0(X_N>t)\leq2e^{-t}$ for~$0\leq t\leq N\log2$. For larger~$t$, the probability is zero because~$X_N\leq N\log2$. Since~$X_N\geq0$, integrating this tail bound yields
\begin{align*}
\E_0e^{vX_N}
&=1+v\int_0^\infty e^{vt}\pr_0(X_N>t)\,dt\\
&\leq1+2v\int_0^\infty e^{-(1-v)t}\,dt\\
&=1+\frac{2v}{1-v}=\frac{3+\lambda_+}{1-\lambda_+}.
\end{align*}
The integral of~$e^{-(1-v)t}$ is finite because~$v<1$. Since~$q_1>1$ and~$(3+\lambda_+)/(1-\lambda_+)>1$, the first factor satisfies
\begin{equation*}
\bigl(\E_0e^{vX_N}\bigr)^{1/q_1}
\leq\left(\frac{3+\lambda_+}{1-\lambda_+}\right)^{1/q_1}
\leq\frac{3+\lambda_+}{1-\lambda_+}.
\end{equation*}
For the second factor, note that~$\cosh\widetilde h\geq1$, and the bound~$\log\cosh y\leq y^2/2$ then gives~$\cosh(q_2\widetilde h)^{N/q_2}\leq e^{Nq_2\widetilde h^{\,2}/2}\leq e^{q_2\rho_0^4/2}$, where the last step uses~$\widetilde h\leq h^2$ and~$Nh^4\leq\rho_0^4$. These estimates together prove~\eqref{eq:total-mass}.

To prove the bound~\eqref{eq:edge-tail}, apply~$r_{N,p}\leq\theta^2\bar r_{N,p}$ to every edge. For a graph with more than~$t$ edges, the product of these factors is at most~$\theta^{2(\lfloor t\rfloor+1)}$. Summing over these graphs and using~\eqref{eq:total-mass} completes the proof.
\end{proof}

In particular, using the above results we have
\(
\E\widehat Z_N^2=\sum_\Gamma\E W_h(\Gamma)^2\leq C_+.
\)
The estimate~\eqref{eq:total-mass} will also be used in the lower tail control of $\hat{Z}_N$. For any~$u\in(0,1)$, \eqref{eq:truncation-isometry} and~\eqref{eq:edge-tail} show that we can approximate~$\widehat Z_N$ with mean-square error at most~$u$ by retaining all graphs with at most~$L$ edges, for a suitable choice of~$L=O(1+\log(1/u))$. Directly listing the retained graphs gives a cost bound~$N^{O(L)}$, which does not establish polynomial dependence jointly on~$N,u^{-1}$. The next two sections choose graph families that can be summed more efficiently for mixed $p$-spin case and SK model respectively.

\section{Mixed $p$-spin models with $p\ge 3$}\label{sec:pure}\label{sec:algorithm}

We truncate by the number of used vertices. Section~\ref{subsec:higher-support} defines the truncation and bounds the discarded part. Section~\ref{subsec:pure-evaluation} evaluates the retained sum by local spin averages. The main estimate must account for both the small second moments of the edges and the number of ways to choose the vertex labels. For example, in the pure three-spin model, a fixed hypergraph shape with~$v$ vertices, all of degree two (if possible), has~$2v/3$ edges. The total contribution from hypergraphs of this shape is~$O(N^{-v/3})$. Lemma~\ref{lem:higher-support-mass} gives a bound after summing over all shapes and labels, including the external field factors. The idea is based on the combinatorial analysis in~\cite{DW26}.

\subsection{Restricting the number of vertices}\label{subsec:higher-support}

For integers~$0\leq k\leq N$ and~$1\leq v\leq N$, define
\begin{equation}\label{eq:higher-retained-sum}
T_{N,k}=\sum_{\substack{\Gamma\subseteq\mathcal E_N\\|V(\Gamma)|\leq k}}W_h(\Gamma),
\qquad M_{N,v}=\sum_{\substack{\Gamma\subseteq\mathcal E_N\\|V(\Gamma)|=v}}
\hh^{2|\partial\Gamma|}\prod_{e\in\Gamma}r_{N,|e|}.
\end{equation}
Here~$T_{N,k}$ is the retained sum, and~$M_{N,v}$ is the sum of expected squared weights over hypergraphs using exactly~$v$ vertices. Lemma~\ref{lem:orthogonality} gives
\begin{equation*}
\E|\widehat Z_N-T_{N,k}|^2=\sum_{v=k+1}^N M_{N,v}.
\end{equation*}
Thus a bound on~$M_{N,v}$ will control the truncation error. Recall~$p_m\geq3$, let~$\rho_1=\max\{1,\rho_0\}$, and define
\begin{equation*}
\kappa_{\min}=\frac{p_m-2}{2p_m},\qquad
C_0=\exp\!\left(1+\sum_{p\in\mathcal P}\frac{\beta_0^2a_p^2}{p!}\right)\rho_1^2,
\qquad C_\star=C_0^{1/\kappa_{\min}}.
\end{equation*}

\begin{lemma}\label{lem:higher-support-mass}
Let~$\alpha\geq1/4$ and define
\begin{equation*}
\kappa_\alpha=\min\left\{\frac{p_m-2}{p_m},\,2\alpha-\frac1{p_m}\right\}.
\end{equation*}
If~$|\hh|\leq\rho_1N^{-\alpha}$, then for every integer~$1\leq v\leq N$,
\begin{equation}\label{eq:higher-support-mass}
M_{N,v}\leq C_0^v(v/N)^{\kappa_\alpha v}
\leq(C_\star v/N)^{\kappa_\alpha v}.
\end{equation}
At~$h=0$, the same bound holds with exponent~$(p_m-2)/p_m$.
\end{lemma}
\begin{proof}
There are two sources of decay. For a fixed number of used vertices, a low total degree forces many odd-degree vertices and therefore many small external field factors. A higher total degree instead gives decay through the edge second moments. We first bound this tradeoff for fixed edge counts, then sum over those counts. Fix~$v$ and nonnegative integers~$(n_p)_{p\in\mathcal P}$. Consider hypergraphs~$\Gamma$ with~$|V(\Gamma)|=v$ and exactly~$n_p$ edges of order~$p$. For edge counts admitting such a graph, the total degree is~$d=\sum_{p\in\mathcal P}pn_p\geq v$. Every used even-degree vertex has degree at least two, so~$d\geq2v-|\partial\Gamma|$. Let~$\ell=(2v-d)_+$, where~$x_+=\max\{x,0\}$. Then~$|\partial\Gamma|\geq\ell$ and~$0\leq\ell\leq v$. Since~$|\hh|\leq1$ and~$|\hh|\leq\rho_1N^{-\alpha}$,
\begin{equation}\label{eq:higher-external-field-factor}
\hh^{2|\partial\Gamma|}
\leq |\hh|^{2\ell}
\leq \rho_1^{2\ell}N^{-2\alpha\ell}
\leq \rho_1^{2v}(v/N)^{2\alpha\ell}.
\end{equation}
The last inequality uses~$\rho_1\geq1$, $\ell\leq v$, and~$v\geq1$.
There are at most~$\binom Nv\prod_p\binom{\binom vp}{n_p}$ choices of vertices and edges. This count may include edge selections that leave some chosen vertices unused, so it is an upper bound. For every graph under consideration, the edge second moments have product~$\prod_p r_{N,p}^{n_p}$. Multiplying the count by this product and the upper bound in~\eqref{eq:higher-external-field-factor} gives a total contribution at most
\begin{align*}
&\binom Nv\rho_1^{2v}(v/N)^{2\alpha\ell}
\prod_{p\in\mathcal P}\binom{\binom vp}{n_p}r_{N,p}^{n_p}\\
&\qquad\leq(eN/v)^v\rho_1^{2v}
(v/N)^{2\alpha(2v-d)_++\sum_p(p-1)n_p}
\prod_{p\in\mathcal P}\frac{(\beta_0^2a_p^2v/p!)^{n_p}}{n_p!}.
\end{align*}
Here we used~\eqref{eq:edge-second-moment}, $\binom Nv\leq(eN/v)^v$, and~$\binom{\binom vp}{n_p}\leq(v^p/p!)^{n_p}/n_p!$. The identity~$v^pN^{1-p}=v(v/N)^{p-1}$ gives the displayed powers of~$v/N$.
We next show that the exponent of~$v/N$ is at least~$\kappa_\alpha v$ for every choice of nonnegative integers~$(n_p)_{p\in\mathcal P}$ satisfying~$\sum_{p\in\mathcal P}pn_p\geq v$. The exponent of~$v/N$ in the preceding bound on the sum of second moments is
\begin{equation*}
-v+\sum_{p\in\mathcal P}(p-1)n_p+2\alpha(2v-d)_+.
\end{equation*}
For every~$p\in\mathcal P$, $(p-1)/p\geq(p_m-1)/p_m$, so~$\sum_{p\in\mathcal P}(p-1)n_p\geq(p_m-1)d/p_m$. This exponent is therefore at least~$g(d)$, where we define
\begin{equation*}
g(d)=-v+\frac{p_m-1}{p_m}d+2\alpha(2v-d)_+.
\end{equation*}
This function is affine on~$[v,2v]$ and increasing after~$2v$. Its minimum is therefore the smaller of~$g(v)=(2\alpha-1/p_m)v$ and~$g(2v)=(p_m-2)v/p_m$, which is~$\kappa_\alpha v$.

We now sum the bounds for graphs with fixed edge counts to estimate~$M_{N,v}$. Since~$v/N\leq1$, the preceding lower bound on the exponent gives
\begin{align*}
M_{N,v}
&\leq e^v\rho_1^{2v}(v/N)^{\kappa_\alpha v}
\prod_{p\in\mathcal P}\left(\sum_{n_p=0}^{\infty}
\frac{(\beta_0^2a_p^2v/p!)^{n_p}}{n_p!}\right)\\
&=e^v\rho_1^{2v}(v/N)^{\kappa_\alpha v}
\exp\!\left(v\sum_{p\in\mathcal P}\frac{\beta_0^2a_p^2}{p!}\right)\\
&=C_0^v(v/N)^{\kappa_\alpha v}.
\end{align*}
In the first inequality, each~$n_p$ ranges over all nonnegative integers. This includes edge counts that cannot occur for a hypergraph with~$v$ used vertices, which only increases the sum because every term is nonnegative. The first equality follows by summing the exponential series for each~$p\in\mathcal P$.

For the second inequality in~\eqref{eq:higher-support-mass}, $C_\star^{\kappa_\alpha}=C_0^{\kappa_\alpha/\kappa_{\min}}\geq C_0$, since~$\kappa_\alpha\geq\kappa_{\min}$ and~$C_0\geq1$. At~$h=0$, only graphs with no odd-degree vertices contribute. They have~$d\geq2v$, which gives the stated exponent.
\end{proof}

The exponent reaches its value at zero external field when~$\alpha\geq(p_m-1)/(2p_m)$. The same transition threshold appears in the results of Dey and Wu~\cite{DW26}*{Corollary~2.4} on free energy fluctuations for mixed~$p$-spin models. For~$h\ne0$, define~$\alpha_h=\log(\rho_1/|\hh|)/\log N$. Then~$|\hh|=\rho_1N^{-\alpha_h}$. The external field assumption~\eqref{eq:field-range} gives~$|\hh|\leq|h|\leq\rho_0N^{-1/4}\leq\rho_1N^{-1/4}$, hence~$\alpha_h\geq1/4$. Applying the lemma with~$\alpha=\alpha_h$ gives the following exponent, which is computable from the supplied~$N,h$.

\begin{equation}\label{eq:field-dependent-kappa}
\kappa=\begin{cases}
\displaystyle\min\left\{\frac{p_m-2}{p_m},\,
\frac{2\log(\rho_1/|\hh|)}{\log N}-\frac1{p_m}\right\},&h\ne0,\\[6pt]
(p_m-2)/p_m,&h=0.
\end{cases}
\end{equation}
Thus~$\kappa_{\min}\leq\kappa\leq(p_m-2)/p_m$. To bound the mean-square truncation error~$\sum_{v=k+1}^{N}M_{N,v}$ given in~\eqref{eq:truncation-isometry}, define
\begin{equation*}
\Phi_N(t)=\kappa t\log\frac{N}{C_\star t}\quad(t>0),
\qquad \Phi_N(0)=0.
\end{equation*}
Lemma~\ref{lem:higher-support-mass} gives~$M_{N,v}\leq e^{-\Phi_N(v)}$, but this upper bound need not be small when~$v$ is comparable to~$N$. Choose a fixed~$\gamma>0$ small enough that~$e^{-\Phi_N(v)}$ decreases geometrically for~$v\leq\gamma N$. We can then bound~$\sum_{k<v\leq\gamma N}M_{N,v}$ by a geometric series. To bound~$\sum_{v>\gamma N}M_{N,v}$, we use the edge-tail estimate~\eqref{eq:edge-tail}, since every hypergraph with more than~$\gamma N$ vertices has more than~$\gamma N/P$ edges. Specifically, choose~$\gamma\in(0,1/4)$ such that
\begin{equation*}
\kappa_{\min}\left(\log\frac{1}{C_\star\gamma}-1\right)\geq2.
\end{equation*}
This ensures
\[
\Phi_N'(t)=\kappa(\log(N/(C_\star t))-1)\geq2 \quad \text{for~$0<t\leq\gamma N$},
\]
giving the required geometric decrease. The upper bound~$1/4$ is only a convenient restriction. Section~\ref{subsec:cutoff-selection} gives an explicit choice of~$\gamma$, depending only on the fixed model and~$\beta_0,\rho_0$.

\begin{proposition}\label{prop:higher-truncation}
Let $\gamma$ be chosen as above, for every integer~$0\leq k<\lfloor\gamma N\rfloor$, we have
\begin{equation}\label{eq:higher-truncation}
\E|\widehat Z_N-T_{N,k}|^2\leq2e^{-\Phi_N(k+1)}+C_+\theta^{2\gamma N/P}.
\end{equation}
\end{proposition}
\begin{proof}
By~\eqref{eq:truncation-isometry}, the error is~$\sum_{v>k}M_{N,v}$. We split this sum at~$\lfloor\gamma N\rfloor$. For~$k<v\leq\lfloor\gamma N\rfloor$, the derivative bound gives~$\Phi_N(v)\geq\Phi_N(k+1)+2(v-k-1)$. Hence
\begin{equation*}
\sum_{v=k+1}^{\lfloor\gamma N\rfloor}M_{N,v}
\leq e^{-\Phi_N(k+1)}\sum_{j\geq0}e^{-2j}
=\frac{e^{-\Phi_N(k+1)}}{1-e^{-2}}
\leq2e^{-\Phi_N(k+1)}.
\end{equation*}
Every remaining graph has more than~$\gamma N$ vertices and hence more than~$\gamma N/P$ edges, since each edge contains at most~$P$ vertices. The edge-tail bound~\eqref{eq:edge-tail} gives the second term.
\end{proof}

\subsection{Exact summation on small vertex sets}\label{subsec:pure-evaluation}

We now evaluate~$T_{N,k}$ without enumerating its hypergraphs. For~$S\subseteq[N]$, define the local spin average
\begin{equation}\label{eq:local-spin-average}
G(S)=\frac{1}{(2\cosh h)^{|S|}}\sum_{\sigma\in\{-1,1\}^{S}}e^{h\sum_{i\in S}\sigma_i}
\prod_{\substack{e\in\mathcal E_N\\e\subseteq S}}(1+\omega_e\sigma_e).
\end{equation}
 Expanding the product gives
\begin{equation*}
G(S)=\sum_{\substack{\Gamma\subseteq\mathcal E_N\\V(\Gamma)\subseteq S}}W_h(\Gamma).
\end{equation*}
Note that~$G(\varnothing)=1$ and~$G([N])=\widehat Z_N$. Summing~$G(S)$ over~$|S|\leq k$ can count the same hypergraph more than once. To correct this overcounting, we add and subtract the values~$G(S)$ using coefficients~$c_{N,k}(|S|)$. For integers~$0\leq s\leq k<N$, define
\begin{equation}\label{eq:exact-support-weight}
c_{N,k}(s)=(-1)^{k-s}\binom{N-s-1}{k-s}.
\end{equation}

The following Proposition~\ref{prop:pure-evaluation} proves that these coefficients cancel the extra occurrences of each graph weight, so every hypergraph using at most~$k$ vertices is counted exactly once.
To evaluate \(G(S)\), we sum over the \(2^{|S|}\) spin configurations on \(S\). For each configuration, we multiply the factors \(1+\omega_e\sigma_e\) over interactions \(e\subseteq S\) and  the probability of that configuration. Finally add the resulting weighted product to the sum. These products are evaluated numerically without expanding them into hypergraph terms. 
Algorithm~\ref{alg:support-sum} sums~$c_{N,k}(|S|)G(S)$ over all~$S\subseteq[N]$ with~$|S|\leq k$.

\begin{algorithm}[!htbp]
\caption{\textsc{SupportSum}$(N,h,\omega,k)$}\label{alg:support-sum}
\begin{algorithmic}[1]
\Require Cached interaction weights~$(\omega_e)$, $p_m\geq3$, external field~$h$, integer~$0\leq k\leq N$.
\If{$k=N$}
  \State \Return $G([N])$ from~\eqref{eq:local-spin-average}.
\EndIf
\State Precompute~$c_{N,k}(s)$ for~$0\leq s\leq k$, and initialize~$T\gets0$.
\For{each $S\subseteq[N]$ with~$|S|\leq k$, generated by cardinality}
  \State Evaluate~$G(S)$ and add~$c_{N,k}(|S|)G(S)$ to~$T$.
\EndFor
\State \Return $T$.
\end{algorithmic}
\end{algorithm}

The identity stated in the following proposition is purely algebraic. For arbitrary finite real edge weights, define~$W_h$ and~$T_{N,k}$ by the same formulas for sums over graphs. No positivity or disorder assumption is needed for this identity.

\begin{proposition}\label{prop:pure-evaluation}
For every finite real weight array, external field~$h\in\mathbb R$, and integer~$0\leq k<N$,
\begin{equation}\label{eq:pure-exact-support}
T_{N,k}=\sum_{\substack{S\subseteq[N]\\|S|\leq k}}c_{N,k}(|S|)G(S).
\end{equation}
For~$k=N$, we have~$T_{N,N}=G([N])$. Algorithm~\ref{alg:support-sum} returns~$T_{N,k}$ exactly for every~$0\leq k\leq N$, using at most
\begin{equation}\label{eq:pure-evaluation-cost}
C_P(N+1)^{P+2}\sum_{s=0}^{k}\binom Ns2^s
\end{equation}
operations, including set generation and weight lookup. The constant~$C_P$ depends only on~$P$.
\end{proposition}
\begin{proof}
For~$k<N$, fix~$U\subseteq[N]$. We first establish the identity
\begin{equation*}
\sum_{S\subseteq U}(-1)^{|U|-|S|}G(S)
=\sum_{V(\Gamma)=U}W_h(\Gamma).
\end{equation*}
Expanding~$G(S)$ and interchanging the finite sums gives
\begin{align*}
\sum_{S\subseteq U}(-1)^{|U|-|S|}G(S)
&=\sum_{S\subseteq U}\sum_{V(\Gamma)\subseteq U}
(-1)^{|U|-|S|}\ind_{\{V(\Gamma)\subseteq S\}}W_h(\Gamma)\\
&=\sum_{V(\Gamma)\subseteq U}W_h(\Gamma)
\sum_{V(\Gamma)\subseteq S\subseteq U}(-1)^{|U|-|S|}.
\end{align*}
For a fixed graph, write~$X=V(\Gamma)$ and~$m=|U\setminus X|$. Since~$S$ must contain~$X$, it can omit vertices only from~$U\setminus X$. Omitting exactly~$j$ of these vertices gives~$\binom mj$ choices of~$S$, each with~$|U|-|S|=j$. Therefore,
\begin{equation*}
\sum_{X\subseteq S\subseteq U}(-1)^{|U|-|S|}
=\sum_{j=0}^{m}\binom mj(-1)^j.
\end{equation*}
If~$m>0$, this sum is zero by the binomial theorem. If~$m=0$, only~$S=U=X$ contributes, giving one. Thus only graphs with~$V(\Gamma)=U$ remain, proving the identity. Summing this identity over~$|U|\leq k$ gives~$T_{N,k}$, since each retained graph appears once. Interchanging the finite sums gives
\begin{align*}
T_{N,k}
=\sum_{\substack{U\subseteq[N]\\|U|\leq k}}
\sum_{S\subseteq U}(-1)^{|U|-|S|}G(S)
=\sum_{\substack{S\subseteq[N]\\|S|\leq k}}G(S)
\sum_{\substack{S\subseteq U\subseteq[N]\\|U|\leq k}}(-1)^{|U|-|S|}.
\end{align*}
For a fixed~$S$, write~$s=|S|$. A superset~$U$ of size~$s+j$ is obtained by adding~$j$ vertices from~$[N]\setminus S$, giving~$\binom{N-s}{j}$ choices, each with sign~$(-1)^j$. The multiplier of~$G(S)$ is therefore~$\sum_{j=0}^{k-s}(-1)^j\binom{N-s}{j}$. Applying Pascal's identity and shifting the index in the second sum yields
\begin{align*}
\sum_{j=0}^{k-s}(-1)^j\binom{N-s}{j}
&=\sum_{j=0}^{k-s}(-1)^j\binom{N-s-1}{j}
-\sum_{j=0}^{k-s-1}(-1)^j\binom{N-s-1}{j}\\
&=(-1)^{k-s}\binom{N-s-1}{k-s}
=c_{N,k}(s).
\end{align*}
The second sum is empty when~$s=k$. Substituting~$c_{N,k}(|S|)$ for the multiplier of~$G(S)$ proves~\eqref{eq:pure-exact-support}. For~$k=N$, both~$G([N])$ and~$T_{N,N}$ sum all graph weights, so they are equal.

It remains to bound the work of evaluating each local average once. Precompute the coefficients by Pascal's recurrence in~$O(N^2)$ operations. A standard fixed-size subset recursion generates all sets of size~$s$ in~$O((N+1)^2\binom Ns)$ operations, including their vertex lists. For each set, enumerate its~$2^s$ spin assignments. Store membership and spins in an array indexed by~$[N]$, then scan the interaction list. It has at most~$C_P(N+1)^P$ entries, each involving at most~$P$ sites. Computing the external field weight and normalization takes another~$O(N)$ operations per assignment. Thus the total cost is at most~$C_P(N+1)^{P+1}$ per assignment. Putting together proves~\eqref{eq:pure-evaluation-cost}.
\end{proof}

Section~\ref{subsec:cutoff-selection} uses the mean-square error bound~\eqref{eq:higher-truncation} and the bound~\eqref{eq:pure-evaluation-cost} on the number of operations needed to compute~$T_{N,k}$ to choose~$k$ so that the requested accuracy is achieved with a polynomial operation count.

\section{The SK model}\label{sec:sk}

Building on the cycle expansion of Aizenman, Lebowitz, and Ruelle~\cite{ALR87}, Dey and Wu~\cite{DW23} identify simple cycles and paths as the leading clusters for~$\widehat Z_N$ at high temperature under weak external fields. For fixed~$\beta>0$, the total contribution of cycles of any fixed length~$\ell\geq3$ remains of order one in~$L^2$ as~$N\to\infty$. Paths of any fixed length~$\ell\geq1$ also have a total contribution of order one when~$h=\rho N^{-1/4}$ with fixed~$\rho\ne0$. Thus graphs using a fixed number of vertices can still make a nonvanishing contribution, so the $o(1)$ estimate in~\eqref{eq:higher-support-mass} is unavailable in SK.

We instead bound the number of \emph{branching vertices}, those of degree at least three, and separately the number of edges. Section~\ref{subsec:sk-support} gives the error control. Sections~\ref{subsec:sk-evaluation} and~\ref{subsec:sk-colorful} decompose graphs with a fixed branching set into paths and cycles, then sum them by color coding. Section~\ref{subsec:sk-estimator} averages the resulting estimates and bounds their error and cost.

\subsection{Branching vertices truncation}\label{subsec:sk-support}

Write~$\omega_{ij}=\omega_{\{i,j\}}$ for distinct~$i,j\in[N]$ and~$\omega_{ii}=0$. Here~$\Gamma\subseteq\mathcal E_{N,2}$ is a selected set of pair interactions, and all its degrees are computed using only these edges. Define its branching set by
\begin{equation*}
\Br(\Gamma)=\{i\in[N]:\deg_\Gamma(i)\geq3\}.
\end{equation*}
For integers~$0\leq D\leq N$ and~$L\geq1$, define
\begin{equation}\label{eq:sk-retained-sum}
T^{\mathrm{SK}}_{N,D,L}=\sum_{\substack{\Gamma\subseteq\mathcal E_N\\|\Br(\Gamma)|\leq D,\ |\Gamma|\leq L}}W_h(\Gamma).
\end{equation}
For~$0\leq b\leq N$, define the sum of expected squared weights over graphs with exactly~$b$ branching vertices by
\begin{align*}
M^{\mathrm{SK}}_{N,b}
&=\sum_{\substack{\Gamma\subseteq\mathcal E_N\\|\Br(\Gamma)|=b}}\E W_h(\Gamma)^2=\sum_{\substack{\Gamma\subseteq\mathcal E_N\\|\Br(\Gamma)|=b}}
\hh^{2|\partial\Gamma|}r_{N,2}^{|\Gamma|}.
\end{align*}
This sum fixes only the number of branching vertices and imposes no edge cutoff. The external field factor is squared. Thus two odd-degree vertices contribute~$\hh^4\leq\rho_0^4/N$ to each expected squared weight.

Figure~\ref{fig:sk-branching-scales} illustrates the second-moment scales for three example graph shapes.
\begin{figure}[!htbp]
\centering
\includegraphics[width=\linewidth]{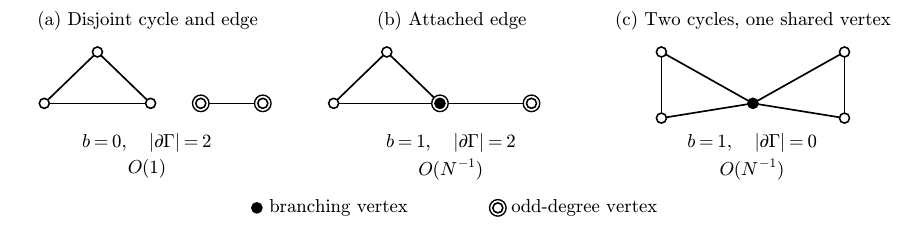}
\caption{Branching and second-moment decay. Here $b=|\Br(\Gamma)|$. Each bound concerns the sum of expected squared weights over all labeled copies of the displayed shape, under $|h|\leq\rho_0N^{-1/4}$. Attaching the separate edge to the triangle removes one free vertex label without changing the two odd-degree vertices, whose squared external field factor is $\widehat h^4=O(N^{-1})$. Two triangles sharing one vertex have one more edge than vertices and also give $O(N^{-1})$.}
\label{fig:sk-branching-scales}
\end{figure}

Recall the constants~$C_+$ and~$\theta$ defined in~\eqref{eq:algorithm-mass-constant} and~\eqref{eq:algorithm-temperature-constants}, respectively, where~$\beta_{\mathrm f}=1$ and~$\Delta=1/2$ for SK. For the branching bound, define
\begin{align}\label{eq:sk-mass-constants}
C_{\mathrm{SK}}=\sqrt{\frac2{1-\beta_0^2}}\exp\!\left(\frac{\rho_0^4\beta_0^2}{1-\beta_0^2}\right),\quad
A_{\mathrm{SK}}=\frac{128e}{(1-\beta_0^2)^2}
\left(\frac{\rho_0^2\beta_0^3}{6}+\frac{\beta_0^4}{24}\right),
\quad C_\star=\max\{1,A_{\mathrm{SK}}\}.
\end{align}
For SK, set~$\kappa=\kappa_{\min}=1$ and define
\begin{equation*}
\Phi_N(t)=t\log\frac{N}{C_\star t}\quad(t>0),
\qquad \Phi_N(0)=0.
\end{equation*}

To estimate the mean-square error from the branching cutoff, we first bound~$M^{\mathrm{SK}}_{N,b}$ for each number~$b$ of branching vertices.

\begin{lemma}\label{lem:sk-branching-mass}
For every integer~$1\leq b\leq N$,
\begin{equation}\label{eq:sk-branching-mass}
M^{\mathrm{SK}}_{N,b}\leq C_{\mathrm{SK}}\left(\frac{A_{\mathrm{SK}}b}{N}\right)^b
\leq C_{\mathrm{SK}}e^{-\Phi_N(b)}.
\end{equation}
\end{lemma}
\begin{proof}
We first count graphs with prescribed degrees, then sum over the permitted degrees. The calculation is related to~\cite{DW23}*{Lemma~4.1}. For any nonnegative integer~$d$, write
\begin{equation*}
\widetilde d=d\bmod2\in\{0,1\}.
\end{equation*}
Fix~$A\subseteq[N]$ with~$|A|=b$. Consider simple graphs with branching set~$A$ and prescribed degrees~$d_1,\ldots,d_N$, and let~$m$ be their number of edges. Then
\begin{equation*}
d_i\geq3\quad(i\in A),\qquad
d_i\in\{0,1,2\}\quad(i\in[N]\setminus A),\qquad
\sum_{i=1}^N d_i=2m.
\end{equation*}
The counting argument in~\cite{FK24}*{Lemma~11.1 and (11.1)} bounds the number of simple graphs with these degrees by
\(
\frac{(2m-1)!!}{\prod_{i=1}^N d_i!}.
\)
Each such graph contributes~$r_{N,2}^{m}\prod_i\hh^{\,2\widetilde d_i}$ to~$M^{\mathrm{SK}}_{N,b}$.

To sum the counting bound over degree sequences, we write it as the expectation of a product with one factor for each vertex. Let~$z$ be an auxiliary standard Gaussian, and write~$\E_z$ for its expectation. Since~$\E_z|z|^{2m}=(2m-1)!!$, the contribution of graphs with this degree sequence is bounded by
\begin{equation*}
\frac{(2m-1)!!\,r_{N,2}^{m}}{\prod_i d_i!}
\prod_i\hh^{\,2\widetilde d_i}
=\E_z\prod_{i=1}^N
\frac{\hh^{\,2\widetilde d_i}r_{N,2}^{d_i/2}|z|^{d_i}}{d_i!}.
\end{equation*}
We now sum over all~$d_i\geq3$ for~$i\in A$ and all~$d_i\in\{0,1,2\}$ for~$i\in[N]\setminus A$, without requiring an even total degree or a sequence realized by a simple graph. This only adds nonnegative terms. Tonelli's theorem allows us to interchange the degree sums with~$\E_z$ and sum separately over the degree at each vertex. To bound these sums at~$x=\sqrt{r_{N,2}}\,|z|$, we use the following inequalities, valid for every~$x\geq0$. For a vertex in~$[N]\setminus A$,
\begin{equation*}
\sum_{d=0}^{2}\frac{\hh^{\,2\widetilde d}x^d}{d!}
=1+\hh^2x+\frac{x^2}{2}
\leq e^{\hh^2x+x^2/2}.
\end{equation*}
For a vertex in~$A$, separating odd and even degrees gives
\begin{align*}
\sum_{d\geq3}\frac{\hh^{\,2\widetilde d}x^d}{d!}
=\hh^2\sum_{j\geq0}\frac{x^{2j+3}}{(2j+3)!}
+\sum_{j\geq0}\frac{x^{2j+4}}{(2j+4)!}
\leq\left(\frac{\hh^2x^3}{6}+\frac{x^4}{24}\right)e^{x^2/2}.
\end{align*}
Here we used~$(2j+3)!\geq6\,2^jj!$ and~$(2j+4)!\geq24\,2^jj!$. Applying the bounds for the two degree sums, using~$(N-b)\hh^2x\leq N\hh^2x$, and summing over the~$\binom Nb$ choices of~$A$, we obtain
\begin{equation*}
M^{\mathrm{SK}}_{N,b}\leq\binom Nb\,\E_z\!\left[
\left(\frac{\hh^2r_{N,2}^{3/2}|z|^3}{6}+\frac{r_{N,2}^{2}z^4}{24}\right)^b
 e^{N\hh^2\sqrt{r_{N,2}}\,|z|+Nr_{N,2}z^2/2}\right].
\end{equation*}

We next estimate this expectation. The bounds~$r_{N,2}\leq\beta_0^2/N$ and~$\hh^2\leq\rho_0^2N^{-1/2}$ imply
\begin{equation*}
\frac{\hh^2r_{N,2}^{3/2}|z|^3}{6}+\frac{r_{N,2}^{2}z^4}{24}
\leq\frac{1}{N^2}
\left(\frac{\rho_0^2\beta_0^3}{6}+\frac{\beta_0^4}{24}\right)(1+z^4),
\end{equation*}
where we used~$|z|^3\leq1+z^4$. To keep the Gaussian integral finite, we bound the exponent by a constant plus a quadratic term whose coefficient is below~$1/2$. Set
\begin{equation*}
\eta=\frac{1-\beta_0^2}{2}\in(0,1/2].
\end{equation*}
Completing the square gives
\begin{align*}
N\hh^2\sqrt{r_{N,2}}\,|z|+\frac{Nr_{N,2}}2z^2
\leq\rho_0^2\beta_0|z|+\frac{\beta_0^2}{2}z^2
\leq\frac{\rho_0^4\beta_0^2}{1-\beta_0^2}+\frac{1-\eta}{2}z^2.
\end{align*}

Multiplying the Gaussian density by~$e^{(1-\eta)z^2/2}$ leaves~$e^{-\eta z^2/2}/\sqrt{2\pi}$. Rescaling the integral therefore yields
\begin{align*}
\E_z\!\left[(1+z^4)^b e^{(1-\eta)z^2/2}\right]
&=\eta^{-1/2}\E_z(1+\eta^{-2}z^4)^b\\
&\leq\eta^{-1/2}2^{b-1}\left(1+\eta^{-2b}(4b-1)!!\right)
\leq\eta^{-1/2}\left(\frac{32b^2}{\eta^2}\right)^b.
\end{align*}
The inequalities use~$(1+t)^b\leq2^{b-1}(1+t^b)$ for~$t\geq0$, the Gaussian moment formula, and~$(4b-1)!!\leq(4b)^{2b}$, together with~$b\geq1$ and~$\eta\leq1$.
Combining these estimates and using the definition of~$C_{\mathrm{SK}}$, we have
\begin{equation*}
M^{\mathrm{SK}}_{N,b}\leq C_{\mathrm{SK}}\binom Nb
\left(\frac{\rho_0^2\beta_0^3/6+\beta_0^4/24}{N^2}\right)^b
\left(\frac{32b^2}{\eta^2}\right)^b.
\end{equation*}
Now~$\binom Nb\leq(eN/b)^b$ and
\(
\frac{32e}{\eta^2}
\left(\frac{\rho_0^2\beta_0^3}{6}+\frac{\beta_0^4}{24}\right)
=A_{\mathrm{SK}},
\)
so
\(
M^{\mathrm{SK}}_{N,b}\leq C_{\mathrm{SK}}\left(\frac{A_{\mathrm{SK}}b}{N}\right)^b.
\)
Finally, $A_{\mathrm{SK}}\leq C_\star$ gives
\begin{equation*}
\left(\frac{A_{\mathrm{SK}}b}{N}\right)^b
\leq\left(\frac{C_\star b}{N}\right)^b
=e^{-\Phi_N(b)}.
\end{equation*}
the proof is complete.
\end{proof}

We now combine the branching bound with the edge-tail estimate. Choose~$\gamma\in(0,1/4)$ with~$\log(1/(C_\star\gamma))-1\geq2$. Then~$\Phi_N'\geq2$ on~$(0,\gamma N]$, so the branching bounds can be summed geometrically in this range. More than~$\gamma N$ branching vertices force more than~$3\gamma N/2$ edges, where we use the edge-tail bound instead.

\begin{proposition}[SK truncation]\label{prop:sk-truncation}
For integers~$0\leq D<\lfloor\gamma N\rfloor$ and~$L\geq1$, let~$T^{\mathrm{SK}}_{N,D,L}$ be the sum in~\eqref{eq:sk-retained-sum}, retaining exactly the graphs with at most~$D$ branching vertices and at most~$L$ original edges. Then
\begin{equation}\label{eq:sk-truncation}
\E|\widehat Z_N-T^{\mathrm{SK}}_{N,D,L}|^2
\leq C_+\theta^{2(L+1)}+2C_{\mathrm{SK}}e^{-\Phi_N(D+1)}+C_+\theta^{3\gamma N}.
\end{equation}
\end{proposition}
\begin{proof}
Let~$B=\lfloor\gamma N\rfloor$. A graph is discarded if it has more than~$L$ edges, or if it has at most~$L$ edges but more than~$D$ branching vertices. These alternatives are disjoint. Orthogonality therefore gives
\begin{align*}
\E|\widehat Z_N-T^{\mathrm{SK}}_{N,D,L}|^2
&=\sum_{|\Gamma|>L}\E W_h(\Gamma)^2
 +\sum_{\substack{|\Gamma|\leq L\\|\Br(\Gamma)|>D}}\E W_h(\Gamma)^2\\
&\leq\sum_{|\Gamma|>L}\E W_h(\Gamma)^2
 +\sum_{b=D+1}^{B}M^{\mathrm{SK}}_{N,b}
 +\sum_{|\Br(\Gamma)|>B}\E W_h(\Gamma)^2.
\end{align*}
In the inequality we removed the edge restriction from the second sum and split it at branching count~$B$. All summands are nonnegative.

The first sum is at most~$C_+\theta^{2(L+1)}$ by~\eqref{eq:edge-tail}. For the middle sum, $D+1\leq b\leq B$ and~$\Phi_N'\geq2$ give
$\Phi_N(b)-\Phi_N(D+1)\geq2(b-D-1)$. Lemma~\ref{lem:sk-branching-mass} consequently yields
\begin{equation*}
\sum_{b=D+1}^{B}M^{\mathrm{SK}}_{N,b}
\leq C_{\mathrm{SK}}e^{-\Phi_N(D+1)}\sum_{j=0}^{B-D-1}e^{-2j}
\leq2C_{\mathrm{SK}}e^{-\Phi_N(D+1)}.
\end{equation*}
For the last sum, the integer inequality~$|\Br(\Gamma)|>B$ implies~$|\Br(\Gamma)|>\gamma N$. Each branching vertex has degree at least three, so
\begin{equation*}
2|\Gamma|=\sum_{i=1}^N\deg_\Gamma(i)
\geq3|\Br(\Gamma)|>3\gamma N.
\end{equation*}
Thus this family is contained in~$\{|\Gamma|>3\gamma N/2\}$, whose total expected squared weight is at most~$C_+\theta^{3\gamma N}$ by~\eqref{eq:edge-tail}. Adding the three estimates proves~\eqref{eq:sk-truncation}.
\end{proof}

\subsection{Augmented graphs}\label{subsec:sk-evaluation}

To evaluate~$T^{\mathrm{SK}}_{N,D,L}$, we group graphs by their exact branching set. We first represent the external field factor by extra auxiliary edges. Adjoin an auxiliary vertex~$0$, assign~$\omega_{0i}=\omega_{i0}=\hh$ for~$i\in[N]$, and define the augmented graph on~$\{0,1,\ldots,N\}$ by
\begin{equation*}
\Gamma^+=\Gamma\cup\bigl\{\{0,i\}:i\in\partial\Gamma\bigr\}.
\end{equation*}
All original edges are kept. The added edges give
\begin{equation*}
\prod_{e\in\Gamma^+}\omega_e=W_h(\Gamma),\qquad
\deg_{\Gamma^+}(i)=\deg_\Gamma(i)+\ind_{\{i\in\partial\Gamma\}}\quad(i\in[N]).
\end{equation*}
The degree in~$\Gamma^+$ of every original vertex is even. The degree at~$0$ is~$|\partial\Gamma|$, which is even by the sum-of-degrees identity. Conversely, evenness at each original vertex uniquely determines whether its edge to~$0$ is present. Thus augmentation is a bijection between simple graphs on~$[N]$ and simple graphs on~$\{0,1,\ldots,N\}$ in which every degree is even, allowing isolated vertices in both cases. An original vertex is branching exactly when its augmented degree is at least four. Only original edges count toward~$L$.

Fix~$A\subseteq[N]$ and define
\begin{equation*}
A^+=A\cup\{0\},\qquad A^c=[N]\setminus A.
\end{equation*}
We call each vertex of~$A^c$ an \emph{outside vertex}. If~$\Br(\Gamma)=A$, then~$\deg_\Gamma(v)\in\{0,1,2\}$ for~$v\in A^c$. Augmentation changes these degrees to~$0,2,2$, respectively, so
\begin{equation}\label{eq:sk-outside-degree}
\deg_{\Gamma^+}(v)=2\qquad\bigl(v\in V(\Gamma^+)\cap A^c\bigr).
\end{equation}
This degree constraint yields a decomposition into the following two types of edge sets.

--\emph{Path pieces}. The collection~$\mathscr P_A$ of \emph{path pieces} consists of the edge sets
\begin{equation*}
\bigl\{\{v_{i-1},v_i\}:1\leq i\leq\ell\bigr\},\qquad \ell\geq1,
\end{equation*}
where~$v_0,\ldots,v_\ell\in\{0,1,\ldots,N\}$ are pairwise distinct and satisfy
\begin{equation*}
v_0,v_\ell\in A^+,\qquad \{v_1,\ldots,v_{\ell-1}\}\subseteq A^c.
\end{equation*}
Thus the endpoints are~$v_0\ne v_\ell$, and the internal vertices are~$v_1,\ldots,v_{\ell-1}$. The case~$\ell=1$ gives the singleton edge set~$\{\{a,b\}\}$ with~$a,b\in A^+$ and~$a\ne b$. We call~$\{a,b\}$ a \emph{direct edge}.

--\emph{Cycle pieces}. The collection~$\mathscr C_A$ of \emph{cycle pieces} consists of the edge sets
\begin{equation*}
\bigl\{\{v_{i-1},v_i\}:1\leq i\leq\ell\bigr\},\qquad
\ell\geq3,\quad v_\ell=v_0,
\end{equation*}
where~$v_0,\ldots,v_{\ell-1}\in\{0,1,\ldots,N\}$ are pairwise distinct and
\begin{equation*}
\bigl|\{v_0,\ldots,v_{\ell-1}\}\cap A^+\bigr|\leq1.
\end{equation*}
Members of~$\mathscr P_A$ and~$\mathscr C_A$ are undirected edge sets. Reversing a path, or changing the starting vertex or orientation of a cycle, gives the same piece. Figure~\ref{fig:sk-augmentation}(c) shows two path pieces and two cycle pieces for~$A=\{1\}$.

\begin{figure}[!htbp]
\centering
\includegraphics[width=\linewidth]{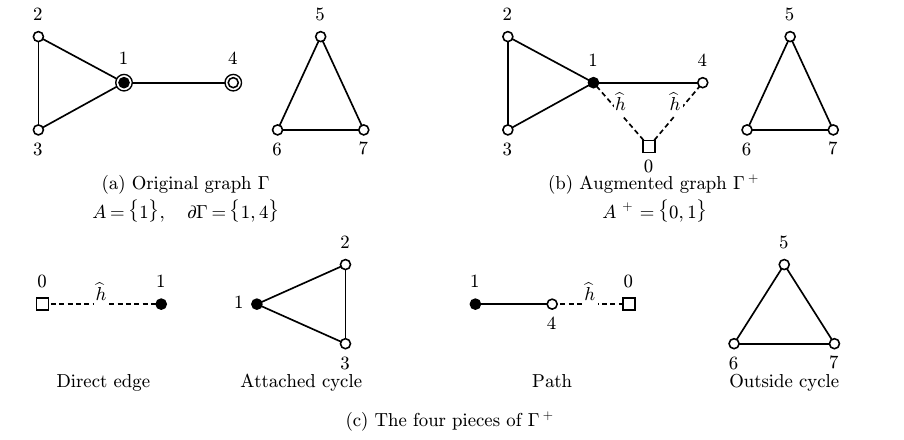}
\caption{Augmentation and the edge partition in Lemma~\ref{lem:sk-decomposition}. In (a), the branching set is $A=\{1\}$ and the odd-degree set is $\partial\Gamma=\{1,4\}$, marked by rings. In (b), two dashed edges to the auxiliary vertex $0$ supply the factor $\widehat h^2$. Row (c) contains exactly the same edges, separated into four pieces. Repeated labels $0$ and $1$ denote the same vertices of $A^+$, not new vertices. The outside vertex sets are disjoint. Only solid, original edges count toward $L$.}
\label{fig:sk-augmentation}
\end{figure}

\begin{lemma}[Path and cycle decomposition]\label{lem:sk-decomposition}
Let~$\Gamma\subseteq\mathcal E_{N,2}$ satisfy~$\Br(\Gamma)=A$. There is a unique collection~$\mathcal Q_\Gamma\subseteq\mathscr P_A\cup\mathscr C_A$ such that
\begin{equation}\label{eq:sk-piece-decomposition}
\Gamma^+=\bigsqcup_{Q\in\mathcal Q_\Gamma}Q,\qquad
V(Q)\cap V(Q')\subseteq A^+\quad
\text{for distinct }Q,Q'\in\mathcal Q_\Gamma.
\end{equation}
The disjoint union means that each edge of~$\Gamma^+$ belongs to exactly one piece. Conversely, let~$\mathcal Q\subseteq\mathscr P_A\cup\mathscr C_A$ satisfy
\begin{equation*}
Q\cap Q'=\varnothing,\qquad V(Q)\cap V(Q')\subseteq A^+
\quad\text{for distinct }Q,Q'\in\mathcal Q,
\end{equation*}
and set~$H=\bigsqcup_{Q\in\mathcal Q}Q$. Then~$H=\Gamma^+$ for a unique~$\Gamma$ with~$\Br(\Gamma)=A$ if and only if
\begin{equation}\label{eq:sk-degree-conditions}
\deg_H(v)\equiv0\pmod2\quad(v\in\{0,1,\ldots,N\}),\qquad
\deg_H(a)\geq4\quad(a\in A).
\end{equation}
\end{lemma}

The first part gives the unique path and cycle decomposition of~$\Gamma^+$ when~$\Br(\Gamma)=A$. The second converse part considers pieces that share no edges and meet only in~$A^+$. It shows that their union equals~$\Gamma^+$ for a graph with~$\Br(\Gamma)=A$ if and only if~\eqref{eq:sk-degree-conditions} holds. These conditions require every degree to be even and each vertex of~$A$ to have degree at least four. In Section~\ref{subsec:sk-colorful}, we sum over these collections and impose the degree conditions. Uniqueness ensures that each graph is counted once.

\begin{proof}
Set
\begin{equation*}
U=V(\Gamma^+)\cap A^c,\qquad
G=\left(U,\Gamma^+\cap\binom{U}{2}\right),
\end{equation*}
and let~$\mathcal K$ be the collection of vertex sets of the connected components of~$G$, including isolated vertices. By~\eqref{eq:sk-outside-degree}, every vertex of~$G$ has degree at most two. Each component is therefore a cycle, a path of positive length, or an isolated vertex. For~$K\in\mathcal K$, define
\begin{equation*}
Q_K=\{e\in\Gamma^+:e\cap K\ne\varnothing\}.
\end{equation*}
If the component on~$K$ is a cycle, its vertices already have degree two within~$K$, so
\begin{equation*}
Q_K=\Gamma^+\cap\binom{K}{2}\in\mathscr C_A.
\end{equation*}
Otherwise, write~$K=\{u_1,\ldots,u_r\}$ in path order, allowing~$r=1$ for an isolated vertex. Equation~\eqref{eq:sk-outside-degree} supplies exactly two edges from~$K$ to~$A^+$, giving
\begin{equation*}
Q_K=\bigl\{\{u_i,u_{i+1}\}:1\leq i<r\bigr\}
\cup\bigl\{\{a,u_1\},\{u_r,b\}\bigr\},
\qquad a,b\in A^+.
\end{equation*}
When~$r=1$, the two neighbors~$a,b$ are distinct because~$\Gamma^+$ is simple. Hence~$Q_K\in\mathscr P_A$ if~$a\ne b$, and~$Q_K\in\mathscr C_A$ if~$a=b$. The latter case necessarily has~$r\geq2$. In Figure~\ref{fig:sk-augmentation}(b), the component sets~$\{2,3\}$, $\{4\}$, and~$\{5,6,7\}$ give, respectively, the cycle~$1,2,3,1$, the path~$1,4,0$, and the cycle~$5,6,7,5$ in part~(c).

Define
\begin{equation*}
\mathcal Q_\Gamma=
\bigl\{\{e\}:e\in\Gamma^+\cap\binom{A^+}{2}\bigr\}
\cup\{Q_K:K\in\mathcal K\}.
\end{equation*}
Every edge either has both endpoints in~$A^+$ or meets exactly one component set~$K$. Also~$V(Q_K)\cap A^c=K$, and distinct component sets are disjoint. These facts prove~\eqref{eq:sk-piece-decomposition}. In Figure~\ref{fig:sk-augmentation}, the only direct edge is~$\{0,1\}$.

For uniqueness, consider any collection satisfying~\eqref{eq:sk-piece-decomposition}. Both edges incident to a vertex~$v\in U$ must belong to the same piece, since~$v$ has degree two in every piece containing it. It follows along each component of~$G$ that all edges of~$Q_K$ belong to one piece. For a piece meeting~$A^c$, the subgraph induced by its vertices in~$A^c$ is connected. Thus this piece cannot contain vertices from two different component sets and is exactly~$Q_K$. Every remaining edge has both endpoints in~$A^+$ and can only be a singleton path piece. Thus~$\mathcal Q_\Gamma$ is unique.

For the converse, the pieces have disjoint edge sets and no common vertex in~$A^c$, so~$H$ is simple and
\begin{equation*}
\deg_H(v)\in\{0,2\}\qquad(v\in A^c).
\end{equation*}
The stated degree conditions are necessary by the augmentation identities. Suppose they hold and put~$\Gamma=H\cap\mathcal E_{N,2}$. For~$i\in[N]$,
\begin{equation*}
\deg_\Gamma(i)=\deg_H(i)-\ind_{\{\{0,i\}\in H\}}.
\end{equation*}
Since~$\deg_H(i)$ is even, this gives~$i\in\partial\Gamma$ if and only if~$\{0,i\}\in H$, and hence~$H=\Gamma^+$. Moreover,
\begin{equation*}
\deg_\Gamma(a)\geq4-1=3\quad(a\in A),\qquad
\deg_\Gamma(v)\leq2\quad(v\in A^c).
\end{equation*}
Thus~$\Br(\Gamma)=A$. The identity~$\Gamma=H\cap\mathcal E_{N,2}$ also proves uniqueness.
\end{proof}

\subsection{Color-coded summation}\label{subsec:sk-colorful}

Fix a candidate branching set~$A\subseteq[N]$ and an edge cutoff~$L\geq1$. Lemma~\ref{lem:sk-decomposition} reduces the graph sum to collections of paths and cycles, but listing all possible vertex sequences of these pieces would be too expensive.

\emph{Color coding} was introduced by Alon, Yuster, and Zwick~\cite{AYZ95} to find simple paths and cycles of a prescribed length in a graph. The difficulty is to extend a path without revisiting a vertex. Their idea is to color vertices randomly and consider paths whose vertices have distinct colors. Extending such a path to an unused color preserves simplicity, so one only needs to remember the colors already used. Here we sum together the weights of paths with the same endpoint and color set. This avoids enumerating their vertex sequences. The number of colors is controlled by~$L$. Cycles are handled similarly, as detailed below.

Lemma~\ref{lem:sk-decomposition} allows different pieces to share vertices in~$A^+$, but not in~$A^c$. We therefore leave~$A^+$ uncolored and fix a coloring~$\chi:A^c\to[2L]$. We use~$2L$ colors because any graph with at most~$L$ original edges uses at most~$2L$ vertices in~$A^c$. A graph or piece is called \emph{colorful} if the vertices it uses in~$A^c$ have distinct colors. To illustrate how we sum path weights, fix~$a\in A^+$ and~$j\in A^c$ with~$\chi(j)=2$. Consider all paths~$a,i,j$ with~$i\in A^c$ and~$\chi(i)=1$. If~$v\in A^c$ has a color other than~$1,2$, every such path extends to a simple path~$a,i,j,v$. Their total weight is
\begin{equation*}
\sum_{\substack{i\in A^c\\\chi(i)=1}}\omega_{ai}\omega_{ij}\omega_{jv}
=\Bigg(\sum_{\substack{i\in A^c\\\chi(i)=1}}\omega_{ai}\omega_{ij}\Bigg)\omega_{jv}.
\end{equation*}
The sum in parentheses is the total weight of the original paths. Multiplying it by~$\omega_{jv}$ computes the total weight of their extensions without processing each path separately.

We combine colorful pieces only if their vertices in~$A^c$ use no common color. This prevents different pieces from sharing a vertex in~$A^c$. For a fixed coloring, we thus count only colorful graphs. In Section~\ref{subsec:sk-estimator}, we divide each included graph's weight by its probability of being colorful under a uniformly random coloring. Averaging over independent colorings gives an unbiased estimate of the full retained sum. Here we compute the colorful graph sums exactly for a fixed coloring~$\chi$.

We group colorful graphs by their branching set, number of edges, and colors used in~$A^c$. The weights of graphs in the same group will be added together.
For~$0\leq k\leq L$ and~$S\subseteq[2L]$, let~$\mathcal G_A^\chi(k,S)$ be the set of colorful graphs~$\Gamma\subseteq\mathcal E_{N,2}$ satisfying
\begin{equation*}
\Br(\Gamma)=A,\qquad |\Gamma|=k,\qquad \chi\bigl(V(\Gamma)\setminus A\bigr)=S.
\end{equation*}
Thus this set contains exactly the colorful graphs with branching set~$A$, $k$ edges, and colors~$S$ on their used vertices in~$A^c$. Define their total weight by
\begin{equation}\label{eq:sk-colorful-identity}
f_A^\chi(k,S):=\sum_{\Gamma\in\mathcal G_A^\chi(k,S)}W_h(\Gamma).
\end{equation}
Summing these quantities over~$k,S$ gives the colorful part of the contribution from~$A$ to~$T^{\mathrm{SK}}_{N,D,L}$ when~$|A|\leq D$.

\begin{proposition}\label{prop:sk-colorful}
Let~$h\in\mathbb R$ be the external field, and let~$\omega_{ij}\in\mathbb R$ be the weight of the edge~$\{i,j\}$ for~$1\leq i<j\leq N$. For any candidate branching set~$A\subseteq[N]$, integer~$L\geq1$, and coloring~$\chi:A^c\to[2L]$, the construction described below computes all sums~$f_A^\chi(k,S)$ in~\eqref{eq:sk-colorful-identity}, for~$0\leq k\leq L$ and~$S\subseteq[2L]$, exactly and deterministically. The total number of operations for all these sums is at most
\begin{equation}\label{eq:sk-colorful-cost}
C_{\mathrm{alg}}(N+2L+1)^6\,16^L\,36^{|A|}\,,
\end{equation}
where~$C_{\mathrm{alg}}$ is an absolute constant.
\end{proposition}

This exact evaluation holds for any external field and any real edge weights. The restriction on the external field~$|h|\leq\rho_0N^{-1/4}$ is used in the truncation and mean-square error estimates. To compute these sums, we first sum the weights of individual path and cycle pieces. We combine pieces only if their vertices in~$A^c$ use no common color. We keep only unions satisfying~\eqref{eq:sk-degree-conditions}. 

\paragraph{Weights of paths and cycles.}
Here we compute the total weights of colorful path and cycle pieces that meet~$A^c$. We group the pieces by the colors of their vertices in~$A^c$ and the vertices they meet in~$A^+$. Recall from Section~\ref{subsec:sk-evaluation} that a \emph{direct edge} joins two distinct vertices of~$A^+$ and therefore has no vertex in~$A^c$. These sums, together with the weights of direct edges, will be used in the next part to compute~$f_A^\chi(k,S)$ by combining pieces.

Give each auxiliary edge~$\{0,i\}$ weight~$\omega_{0i}=\omega_{i0}=\hh=\tanh h$. The weight of a path or cycle is the product of its edge weights. For~$a\in A^+$, $j\in A^c$, and~$S\subseteq[2L]$, define~$P_a(S,j)$ to be the sum of weights of paths from~$a$ to~$j$ whose vertices after~$a$ lie in~$A^c$ and have exactly the colors in~$S$, each once. The preceding example with two colors is the entry~$P_a(\{1,2\},j)$.

The initial entries are~$P_a(\{\chi(j)\},j)=\omega_{aj}$. Entries with~$S=\varnothing$ or~$\chi(j)\notin S$ are zero. For~$|S|\geq2$ and~$\chi(j)\in S$, compute
\begin{equation}\label{eq:sk-path-recurrence}
P_a(S,j)=\sum_{\substack{i\in A^c\\\chi(i)\in S\setminus\{\chi(j)\}}}
P_a(S\setminus\{\chi(j)\},i)\omega_{ij}.
\end{equation}
Deleting the last vertex~$j$ of a path removes its color and the edge from its penultimate vertex~$i$. Conversely, appending~$j$ to a path using~$S\setminus\{\chi(j)\}$ cannot repeat a vertex. This proves the recurrence by induction on~$|S|$. For each fixed starting vertex~$a\in A^+$, the values~$P_a(S,j)$ form a table indexed by color sets~$S\subseteq[2L]$ and endpoints~$j\in A^c$. This table has~$2^{2L}|A^c|\leq N2^{2L}$ entries.

The table~$P_a(S,j)$ only counts paths starting at~$a\in A^+$. To construct cycles contained entirely in~$A^c$, we also need paths whose starting vertex~$r$ lies in~$A^c$. We therefore introduce a second path table. For~$r,j\in A^c$, let~$P_r^\circ(S,j)$ sum the weights of paths from~$r$ to~$j$ contained in~$A^c$, with vertex colors exactly~$S$, each once. Thus~$S$ includes~$\chi(r)$. Initialize~$P_r^\circ(\{\chi(r)\},r)=1$, corresponding to the path with no edges. All other singleton entries are zero. For~$|S|\geq2$, $\chi(r)\in S$, and~$\chi(j)\in S\setminus\{\chi(r)\}$, compute
\begin{equation}\label{eq:sk-cycle-recurrence}
P_r^\circ(S,j)=\sum_{\substack{i\in A^c\\\chi(i)\in S\setminus\{\chi(j)\}}}
P_r^\circ(S\setminus\{\chi(j)\},i)\omega_{ij}.
\end{equation}
All other entries are zero. Deleting or appending the last vertex gives the same induction as for~\eqref{eq:sk-path-recurrence}. Both tables are computed in increasing order of~$|S|$.

Both tables sum weights of paths ending in~$A^c$. Lemma~\ref{lem:sk-decomposition} requires path pieces to have two distinct endpoints in~$A^+$. Its cycle pieces meet~$A^+$ in at most one vertex. To obtain a path piece from a path counted by~$P_a(S,j)$, we append the edge~$\{j,a'\}$ with~$a'\in A^+\setminus\{a\}$. To obtain a cycle, we add~$\{j,a\}$ to a path counted by~$P_a(S,j)$, or~$\{j,r\}$ to a path counted by~$P_r^\circ(S,j)$. The size restrictions below ensure that each resulting cycle has at least three edges. For~$S\subseteq[2L]$, define the following sums.
\begin{itemize}
\item For~$a<a'$ in~$A^+$, let~$p_{aa'}(S)$ be the total weight of colorful paths between~$a$ and~$a'$ with nonempty interior in~$A^c$ and interior color set~$S$.
\item For~$a\in A^+$, let~$q_a(S)$ be the total weight of colorful cycles meeting~$A^+$ only at~$a$, with all other vertices in~$A^c$ and their colors forming~$S$.
\item Let~$\zeta(S)$ be the total weight of colorful cycles contained in~$A^c$ and having outside color set~$S$.
\end{itemize}
Multiplying by the added edge weights, summing the table entries, and correcting for repeated counts gives
\begin{align}
p_{aa'}(S)&=\sum_{j\in A^c}P_a(S,j)\omega_{ja'},&&a<a'\text{ in }A^+,\ |S|\geq1,\label{eq:sk-path-close}\\
q_a(S)&=\frac12\sum_{j\in A^c}P_a(S,j)\omega_{ja},&&a\in A^+,\ |S|\geq2,\label{eq:sk-return-close}\\
\zeta(S)&=\frac1{2|S|}\sum_{\substack{r\in A^c\\\chi(r)\in S}}\sum_{j\in A^c}P_r^\circ(S,j)\omega_{jr},&&|S|\geq3.\label{eq:sk-cycle-close}
\end{align}
Set these sums to zero outside the stated size ranges. Starting at the smaller endpoint counts each undirected path once. A cycle through~$a$ is counted in two orientations, while a cycle contained in~$A^c$ is counted in two orientations from each of its~$|S|$ roots. The denominators remove these multiplicities. The size restrictions exclude traversing a single edge twice.

\paragraph{Combining the pieces.}
The preceding part gives weight sums for individual path and cycle pieces, whereas our target~$f_A^\chi(k,S)$ sums weights of whole graphs. The second part of Lemma~\ref{lem:sk-decomposition} specifies which collections of pieces form the required augmented graphs. For each such collection, the product of its piece weights equals the weight of the corresponding original graph.

For example, if~$a\in A$ and~$S\cap T=\varnothing$, then~$q_a(S)q_a(T)$ sums the weights of unions of two cycles whose only common vertex is~$a$, as in Figure~\ref{fig:sk-color-tables}(a). In general, we must exclude shared vertices in~$A^c$ and unions with more than~$L$ original edges. We must also check the final degrees and count each graph once. We proceed in three steps.
\begin{enumerate}
\item Combine the path and cycle weight sums using disjoint outside color sets, and discard collections with more than~$L$ original edges.
\item Correct for the different orders in which the same pieces can be selected. Then add direct edges within~$A^+$, selecting each at most once.
\item After all pieces and direct edges have been chosen, keep only collections satisfying the degree conditions~\eqref{eq:sk-degree-conditions} in Lemma~\ref{lem:sk-decomposition}.
\end{enumerate}

We implement this construction using polynomial multiplication. Introduce formal variables~$t$, $x_i$ for~$i\in[2L]$, and~$y_a$ for~$a\in A^+$, and write~$x_S=\prod_{i\in S}x_i$. A colorful piece~$Q$ with~$k$ original edges and outside color set~$S$ contributes the weighted monomial
\begin{equation*}
\left(\prod_{e\in Q}\omega_e\right)t^kx_S\prod_{a\in A^+}y_a^{\deg_Q(a)}\,.
\end{equation*}
Its coefficient is the piece's weight. The exponents record the original edge count, outside colors, and degrees at~$A^+$. These degrees include auxiliary edges.

Let~$\Psi_A$ sum these monomials over colorful pieces containing a vertex in~$A^c$. We construct it directly from the previously computed sums:
\begin{equation}\label{eq:sk-piece-polynomial}
\begin{array}{@{}p{0.40\linewidth}ll@{}}
Piece & \text{Contribution to }\Psi_A & \text{Range}\\ \hline
Cycle in~$A^c$ & \zeta(S)t^{|S|}x_S & |S|\geq3\\
Path between~$a<a'$ in~$A$ & p_{aa'}(S)t^{|S|+1}x_Sy_ay_{a'} & |S|\geq1\\
Cycle through~$a\in A$, with all other vertices in~$A^c$ & q_a(S)t^{|S|+1}x_Sy_a^2 & |S|\geq2\\
Path from~$0$ to~$a\in A$ & p_{0a}(S)t^{|S|}x_Sy_0y_a & |S|\geq1\\
Cycle through~$0$, with all other vertices in~$A^c$ & q_0(S)t^{|S|-1}x_Sy_0^2 & |S|\geq2
\end{array}
\end{equation}
Sum these contributions over the indicated endpoints and color sets~$S\subseteq[2L]$. The powers of~$t$ count only original edges; in particular, a cycle through~$0$ has two auxiliary edges and hence~$|S|-1$ original edges. Multiplying piece monomials multiplies their weights and adds their exponents. To exclude repeated outside colors and more than~$L$ original edges, we impose
\begin{equation}\label{eq:sk-array-relations}
x_i^2=0\quad(i\in[2L]),\qquad t^{L+1}=0\,.
\end{equation}
The first relation excludes pieces sharing an outside vertex. It also excludes distinct outside vertices with the same color, because the desired graph must be colorful, see Figure~\ref{fig:sk-color-tables} for an example. The second relation enforces the original edge cutoff.

\begin{figure}[!htbp]
\centering
\includegraphics[width=\linewidth]{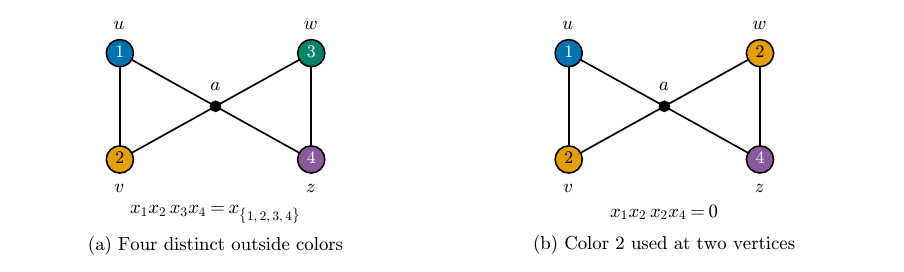}
\caption{The color test when combining two cycle pieces. Both cases show the same graph, with $A=\{a\}$ and $L\geq6$. Vertex names are outside the circles, and the numbers inside are their colors. The isolated auxiliary vertex $0$ is omitted. In (a), all four outside colors are distinct, so the two cycles can be combined. In (b), the distinct vertices $v$ and $w$ both have color $2$. Each cycle is colorful by itself, but their product contains $x_2^2$ and is discarded. The graph itself is valid. Section~\ref{subsec:sk-estimator} corrects this coloring restriction by reweighting and averaging.}
\label{fig:sk-color-tables}
\end{figure}

The expansion of~$\Psi_A^j$ gives ordered choices of~$j$ pieces. Since every piece uses a nonempty outside color set, the color restriction prevents repeated selection of a piece. Each permitted collection therefore appears in exactly~$j!$ orders. We count it once by defining
\begin{equation*}
E_A=\sum_{j=0}^{L}\frac{\Psi_A^j}{j!}\,.
\end{equation*}
Every piece contains at least one original edge, so collections with at most~$L$ original edges contain at most~$L$ pieces. So far, the exponent of~$y_a$ records the exact degree at~$a$. We now replace these exponents by degree states to reduce the number of coefficients stored.

At a vertex of~$A$, the final augmented degree must be even and at least four. Once the current degree is at least four, further additions cannot lower it below four, so only parity matters for the final check. We keep degrees~$0,1,2,3$ separately to determine whether subsequent additions reach the required threshold. For a union~$H$ of selected pieces and direct edges, define the \emph{degree state}
\begin{equation*}
d_a=\degstate\bigl(\deg_H(a)\bigr),\qquad
\degstate(d)=\begin{cases}
d,&0\leq d\leq3,\\
4,&d\geq4\text{ and }d\text{ is even},\\
5,&d\geq5\text{ and }d\text{ is odd},
\end{cases}
\end{equation*}
for nonnegative integers~$d$.

From now on, the factor~$y_a^{d_a}$ records this degree state. After each multiplication, replace~$y_a^d$ by~$y_a^{\degstate(d)}$, leaving the coefficient unchanged, and add coefficients of terms with the same monomial. Subsequent updates can use the stored states because
\begin{equation*}
\degstate\!\left(\degstate(d)+\degstate(d')\right)=\degstate(d+d')\,.
\end{equation*}
Thus each degree exponent has only six possible values. Algebraically, this replacement is implemented by simply setting
\begin{equation}\label{eq:sk-degree-reduction}
y_a^6=y_a^4\qquad(a\in A^+)\,.
\end{equation}

The polynomial~$E_A$ does not yet include direct edges, whose endpoints both lie in~$A^+$. Such edges use no outside colors, so the relations~$x_i^2=0$ cannot prevent their repeated selection. Consider an original direct edge~$\{a,a'\}\subseteq A$. To account for collections both with and without this edge, let
\begin{equation*}
E_A+t\omega_{aa'}y_ay_{a'}E_A=(1+t\omega_{aa'}y_ay_{a'})E_A\,.
\end{equation*}
The first summand keeps the collections unchanged. In the second, every collection gains the edge~$\{a,a'\}$: its weight is multiplied by~$\omega_{aa'}$, its original edge count increases by one, and the two endpoint exponents increase by one before degree reduction. Apply this operation to the current polynomial once for each original direct edge. For an auxiliary edge~$\{0,a\}$, use~$1+\omega_{0a}y_0y_a$ instead, in the way, adding it does not increase the original edge count. Processing each edge once prevents repeated selection. The resulting polynomial is
\begin{equation}\label{eq:sk-assembled-polynomial}
F_A=\prod_{\substack{a,a'\in A\\a<a'}}(1+t\omega_{aa'}y_ay_{a'})
\prod_{a\in A}(1+\omega_{0a}y_0y_a)E_A\,.
\end{equation}
Apply the color, edge-count, and degree reductions after every multiplication.

To select graphs whose branching set is exactly~$A$, Lemma~\ref{lem:sk-decomposition} requires the completed union~$H$ to have even degree at least four at every vertex of~$A$, and even degree at~$0$. Every used vertex in~$A^c$ already has degree two. We therefore select coefficients with
\begin{equation*}
d_a=4\quad(a\in A),\qquad d_0\in\{0,2,4\}\,.
\end{equation*}
State~$4$ includes degrees~$4,6,8,\ldots$; state~$5$ represents odd degrees and is excluded only at this final selection. These are augmented degrees, an original degree of five becomes six after augmentation and has state~$4$. For a reduced polynomial~$F$ and a monomial~$M$, let~$\operatorname{coeff}(F;M)$ denote the coefficient of~$M$ in~$F$. We claim that
\begin{equation}\label{eq:sk-accepting-coefficient}
f_A^\chi(k,S)=\sum_{d_0\in\{0,2,4\}}
\operatorname{coeff}\!\left(F_A;\,t^kx_Sy_0^{d_0}\prod_{a\in A}y_a^4\right)\,.
\end{equation}
The extracted coefficients sum the weights with the specified edge count, color set, and degree states. 

For computation, store each polynomial by its coefficients, indexed by~$0\leq k\leq L$, $S\subseteq[2L]$, and~$d_a\in\{0,\ldots,5\}$ for~$a\in A^+$. Writing~$b=|A|$, this requires
\begin{equation}\label{eq:sk-array-size}
K_{b,L}=(L+1)2^{2L}6^{b+1}
\end{equation}
entries. Compute~$Y_0=1$ and~$Y_j=Y_{j-1}\Psi_A/j$ for~$1\leq j\leq L$, and the sum~$E_A=\sum_{j=0}^{L}Y_j$, and then perform the direct-edge updates. Each multiplication of two coefficients computes all pairwise products of the weights they represent.

\begin{proof}[Proof of Proposition~\ref{prop:sk-colorful}]
 For a piece~$Q$, write
\begin{equation*}
m(Q)=|Q\cap\mathcal E_{N,2}|,\qquad
C_\chi(Q)=\{\chi(v):v\in V(Q)\cap A^c\}.
\end{equation*}
Let~$\mathcal Q\subseteq\mathscr P_A\cup\mathscr C_A$ consist of the colorful pieces~$Q$ that contain at least one vertex in~$A^c$ and satisfy~$m(Q)\leq L$. The path recurrences and formulas~\eqref{eq:sk-path-close}--\eqref{eq:sk-cycle-close} count each such piece once, so
\begin{equation*}
\Psi_A=\sum_{Q\in\mathcal Q}
\left(\prod_{e\in Q}\omega_e\right)
t^{m(Q)}x_{C_\chi(Q)}\prod_{a\in A^+}y_a^{\deg_Q(a)}.
\end{equation*}
This expansion verifies the coefficients computed from the path tables. The algorithm does not enumerate~$\mathcal Q$.

Expanding~$\Psi_A^j$ gives ordered choices of~$j$ pieces. The relations~$x_i^2=0$ and~$t^{L+1}=0$ discard choices with a repeated color on vertices in~$A^c$ or more than~$L$ original edges. The remaining pieces have nonempty, pairwise disjoint color sets and are therefore distinct. Each collection appears in exactly~$j!$ orders, so division by~$j!$ counts it once. Each piece has at least one original edge. Hence a collection with at most~$L$ original edges contains at most~$L$ pieces and occurs in~$E_A$.

The direct edge factors select a subset~$\mathcal D\subseteq\binom{A^+}{2}$. Together with a retained collection~$\mathcal C\subseteq\mathcal Q$, they give
\begin{equation*}
H=\mathcal D\cup\bigcup_{Q\in\mathcal C}Q.
\end{equation*}
Distinct pieces have disjoint vertex sets in~$A^c$, and each of their edges has at least one endpoint in~$A^c$. They therefore share no edges with one another or with~$\mathcal D$, so~$H$ is simple and degrees add. If~$|H\cap\mathcal E_{N,2}|\leq L$, its contribution is
\begin{equation*}
\left(\prod_{e\in H}\omega_e\right)
t^{|H\cap\mathcal E_{N,2}|}
x_{\chi(V(H)\cap A^c)}
\prod_{a\in A^+}y_a^{\degstate(\deg_H(a))}\,.
\end{equation*}
Otherwise it vanishes.

The coefficient selection in~\eqref{eq:sk-accepting-coefficient} requires even degree at~$0$ and even degree at least four at every~$a\in A$. Every used vertex in~$A^c$ has degree two, so~$H$ satisfies~\eqref{eq:sk-degree-conditions}. The second part of Lemma~\ref{lem:sk-decomposition} then gives
\begin{equation*}
\Gamma=H\cap\mathcal E_{N,2},\qquad H=\Gamma^+,\qquad
\Br(\Gamma)=A,\qquad \prod_{e\in H}\omega_e=W_h(\Gamma).
\end{equation*}
A vertex in~$A^c$ loses at most one edge when~$0$ is deleted, so~$V(H)\cap A^c=V(\Gamma)\setminus A$. Each contribution selected by~$t^kx_S$ therefore gives a graph~$\Gamma\in\mathcal G_A^\chi(k,S)$ with weight~$W_h(\Gamma)$. Conversely, each~$\Gamma\in\mathcal G_A^\chi(k,S)$ has the unique decomposition of Lemma~\ref{lem:sk-decomposition}. At most~$k\leq L$ pieces in this decomposition contain at least one vertex in~$A^c$. The contribution of~$\Gamma$ satisfies the color, edge count, and degree restrictions and is counted once. This proves~\eqref{eq:sk-accepting-coefficient}, including the empty graph when~$A=\varnothing$, $k=0$, and~$S=\varnothing$. The argument counts contributions before summing weights, so it also covers zero and negative weights.

For the operation bound, let~$b=|A|$ and~$\tilde N=N+2L+1$. Store pair weights in an array and color sets as~$2L$ binary flags. There are~$b+1$ tables~$P_a$ and~$N-b$ tables~$P_r^\circ$, hence~$N+1$ tables in total. Each has at most~$N4^L$ entries, and each recurrence has at most~$N$ summands. Computing all tables therefore requires evaluating at most~$(N+1)N^2 4^L$ summands. Including color tests, set generation, indexing, and the sums in~\eqref{eq:sk-path-close}--\eqref{eq:sk-cycle-close} costs~$O(\tilde N^6 4^L)$.

Each coefficient array multiplication examines at most
\(
K_{b,L}^2=36(L+1)^2\,16^L\,36^b
\)
pairs of entries, using~$O(\tilde N^2)$ operations per pair for color tests, degree reduction, and indexing. Computing~$E_A$ and multiplying the direct edge factors requires at most~$L+(b+1)^2$ such multiplications. Since~$L+1\leq\tilde N$ and~$L+(b+1)^2\leq2\tilde N^2$, the total cost is
\begin{equation*}
O\!\left((L+(b+1)^2)\tilde N^2 K_{b,L}^2\right)
=O\!\left(\tilde N^6\,16^L\,36^b\right).
\end{equation*}
Input preparation, initialization, scalar operations, and coefficient extraction fit the same bound, proving~\eqref{eq:sk-colorful-cost}.
\end{proof}

\subsection{Unbiased estimation and cost}\label{subsec:sk-estimator}

Proposition~\ref{prop:sk-colorful} computes the colorful part exactly. A plain average over random colorings would multiply each graph's weight by its probability of being colorful. We correct this bias by dividing each included contribution by its inclusion probability. We then average independent trials to reduce the error.

Fix integers~$0\leq D\leq N$ and~$L,R\geq1$, where~$R$ is the number of trials for each candidate branching set. We use~$\E_J$, $\E_\chi$, and~$\E_{J,\chi}$ for disorder, color, and joint expectations, respectively. For a fixed candidate~$A$, draw~$\chi:A^c\to[2L]$ with independent uniform colors. A fixed set of~$0\leq m\leq2L$ outside vertices receives distinct colors with probability
\begin{equation}\label{eq:color-inclusion}
\pi_{m,2L}=\frac{(2L)_m}{(2L)^m},\qquad
(2L)_m=(2L)(2L-1)\cdots(2L-m+1),\qquad \pi_{0,2L}=1.
\end{equation}
Every graph with at most~$L$ edges has at most~$2L$ outside vertices, so these probabilities are positive. For fixed~$J$, consider a retained graph with branching set~$A$ and~$m$ outside vertices. Its coefficient is one when the graph is colorful and zero otherwise, with mean~$\pi_{m,2L}$. Dividing that coefficient by~$\pi_{m,2L}$ makes its mean one. This corrects each signed graph weight, without requiring it to be positive.

The correction is computable from the stored indices. Every graph in~$\mathcal G_A^\chi(k,S)$ has exactly~$|S|$ used vertices in~$A^c$, so all these graphs have the same inclusion probability~$\pi_{|S|,2L}$. We divide the whole entry by this probability, without identifying its individual graphs. Summing the corrected entries gives, for one trial and fixed~$J$,
\begin{equation*}
\E_\chi\!\left[\sum_{k=0}^{L}\sum_{S\subseteq[2L]}
\frac{f_A^\chi(k,S)}{\pi_{|S|,2L}}\right]
=\sum_{\substack{\Gamma\subseteq\mathcal E_N\\\Br(\Gamma)=A,\ |\Gamma|\leq L}}W_h(\Gamma).
\end{equation*}
The right side is the contribution of the candidate~$A$ to~$T^{\mathrm{SK}}_{N,D,L}$. For each~$A\subseteq[N]$ with~$|A|\leq D$ and each trial~$n\in\{1,\ldots,R\}$, let~$\chi_{A,n}:A^c\to[2L]$ assign independent uniform colors to the vertices. Take these colorings independently of one another and of~$J$. Define the \emph{randomized estimate} of~$T^{\mathrm{SK}}_{N,D,L}$ by
\begin{equation}\label{eq:sk-estimator-definition}
\widetilde T^{\mathrm{SK}}_{N,D,L,R}
:=\sum_{\substack{A\subseteq[N]\\|A|\leq D}}
\frac1R\sum_{n=1}^{R}\sum_{k=0}^{L}\sum_{S\subseteq[2L]}
\frac{f_A^{\chi_{A,n}}(k,S)}{\pi_{|S|,2L}}\,.
\end{equation}
Thus we average~$R$ corrected color sums for each candidate~$A$, then add the results over all candidates. Algorithm~\ref{alg:sk-sum} computes this quantity.

\begin{algorithm}[!htbp]
\caption{\textsc{SKSum}$(N,h,\omega,D,L,R)$}\label{alg:sk-sum}
\begin{algorithmic}[1]
\Require Cached pair weights~$(\omega_{ij})$, external field~$h$, integers~$0\leq D\leq N$ and~$L,R\geq1$.
\State Compute~$\hh\gets\tanh h$, set~$\omega_{0i}=\omega_{i0}\gets\hh$ for~$i\in[N]$, compute~$\pi_{m,2L}$ for~$0\leq m\leq2L$, and initialize~$T\gets0$.
\For{each $A\subseteq[N]$ with $|A|\leq D$, generated by cardinality}
  \State $A^+\gets A\cup\{0\}$ and~$T_A\gets0$.
  \For{$n=1,\ldots,R$}
    \State Draw~$\chi:A^c\to[2L]$ with independent uniform colors, independently of~$J$ and all previous colorings.
    \State Initialize the path tables and coefficient arrays to zero for this trial.
    \State Compute~$P_a,P_r^\circ$ in increasing~$|S|$, using their initial states and~\eqref{eq:sk-path-recurrence}--\eqref{eq:sk-cycle-recurrence}.
    \State Compute~$p_{aa'},q_a,\zeta$ by~\eqref{eq:sk-path-close}--\eqref{eq:sk-cycle-close}, then form~$\Psi_A$ by~\eqref{eq:sk-piece-polynomial}.
    \State $Y\gets1$ and~$E\gets1$.
    \For{$j=1,\ldots,L$}
      \State $Y\gets Y\Psi_A/j$ and then~$E\gets E+Y$, using~\eqref{eq:sk-array-relations} and~\eqref{eq:sk-degree-reduction}.
    \EndFor
    \State Form~$F_A$ from~$E$ using the factors for direct edges in~\eqref{eq:sk-assembled-polynomial}.
    \State Extract~$f_A^\chi(k,S)$ by~\eqref{eq:sk-accepting-coefficient} and add~$\sum_{k=0}^L\sum_{S\subseteq[2L]}\pi_{|S|,2L}^{-1}f_A^\chi(k,S)$ to~$T_A$.
  \EndFor
  \State $T\gets T+T_A/R$.
\EndFor
\State \Return $T$.
\end{algorithmic}
\end{algorithm}

\begin{proposition}\label{prop:sk-estimator}
For integers~$0\leq D\leq N$ and~$L,R\geq1$, the estimate defined in~\eqref{eq:sk-estimator-definition} and computed by Algorithm~\ref{alg:sk-sum} satisfies
\begin{equation}\label{eq:sk-unbiased}
\E_\chi[\widetilde T^{\mathrm{SK}}_{N,D,L,R}\mid J]=T^{\mathrm{SK}}_{N,D,L}
\end{equation}
for every fixed realization of~$J$. Under~\eqref{eq:disorder},
\begin{equation}\label{eq:sk-evaluation-error}
\E_{J,\chi}|\widetilde T^{\mathrm{SK}}_{N,D,L,R}-T^{\mathrm{SK}}_{N,D,L}|^2\leq\frac{C_+e^{2L}}{R}.
\end{equation}
If, in addition, \(D<\lfloor\gamma N\rfloor\), then 
\begin{equation}\label{eq:sk-complete-error}
\E|\widetilde T^{\mathrm{SK}}_{N,D,L,R}-\widehat Z_N|^2
\leq C_+\theta^{2(L+1)}+2C_{\mathrm{SK}}e^{-\Phi_N(D+1)}+C_+\theta^{3\gamma N}+\frac{C_+e^{2L}}R.
\end{equation}
For an absolute constant~$C_{\mathrm{alg}}$, the operation count on every finite input and every color outcome is at most
\begin{equation}\label{eq:sk-evaluation-cost}
C_{\mathrm{alg}}R(N+2L+1)^6 16^L\sum_{b=0}^{D}\binom Nb36^b.
\end{equation}
\end{proposition}

\begin{proof}
The probability correction gives unbiasedness for fixed~$J$, but rare inclusions can have large corrected coefficients. We first bound the variance of each coefficient after~$R$ trials. We then use disorder orthogonality to control the error of the complete signed sum.
Fix a retained graph~$\Gamma$ and let~$m=|V(\Gamma)\setminus\Br(\Gamma)|$. Let~$\ind_n(\Gamma)$ indicate that its outside colors are distinct in trial~$n$ for~$A=\Br(\Gamma)$. This is the only candidate set at which~$\Gamma$ can appear, and its coefficient in the output is
\begin{equation*}
c_\Gamma=\frac1R\sum_{n=1}^R\frac{\ind_n(\Gamma)}{\pi_{m,2L}}.
\end{equation*}
The indicators are independent across trials and each has mean~$\pi_{m,2L}$. The bound~$n!\geq(n/e)^n$ gives~$\pi_{m,2L}\geq\pi_{2L,2L}=(2L)!/(2L)^{2L}\geq e^{-2L}$. Therefore
\begin{equation}\label{eq:sk-coefficient-variance}
\E_\chi c_\Gamma=1,\qquad
\E_\chi(c_\Gamma-1)^2=\frac{1-\pi_{m,2L}}{R\pi_{m,2L}}\leq\frac{e^{2L}}{R}.
\end{equation}
Summing the first identity against the fixed graph weights proves~\eqref{eq:sk-unbiased}. To estimate the error, first write the difference of the two graph sums as
\begin{equation*}
\widetilde T^{\mathrm{SK}}_{N,D,L,R}-T^{\mathrm{SK}}_{N,D,L}
=\sum_{\substack{\Gamma\subseteq\mathcal E_N\\|\Br(\Gamma)|\leq D,\ |\Gamma|\leq L}}
(c_\Gamma-1)W_h(\Gamma).
\end{equation*}
Graphs with the same candidate branching set are tested using the same random coloring, so their coefficients can be dependent. Their covariance terms need not vanish when~$J$ is held fixed. Instead, average also over~$J$. The full array~$(c_\Gamma-1)$ is independent of~$J$, and Lemma~\ref{lem:orthogonality} removes the cross terms between different graph weights. Applying it to this sum gives
\begin{align*}
\E_{J,\chi}|\widetilde T^{\mathrm{SK}}_{N,D,L,R}-T^{\mathrm{SK}}_{N,D,L}|^2
&=\sum_{\substack{|\Br(\Gamma)|\leq D\\|\Gamma|\leq L}}
\E_\chi(c_\Gamma-1)^2\E_J W_h(\Gamma)^2\\
&\leq\frac{e^{2L}}R\sum_{\Gamma\subseteq\mathcal E_N}\E_J W_h(\Gamma)^2
\leq\frac{C_+e^{2L}}R,
\end{align*}
where the last step uses~\eqref{eq:total-mass} and~$r_{N,2}\leq\bar r_{N,2}$.

To obtain the total approximation error, separate the error in evaluating the retained sum from the error of truncating the expansion. Given~$J$, the first has color mean zero by~\eqref{eq:sk-unbiased}, while the second is fixed. The cross term therefore vanishes, giving
\begin{align*}
\E_{J,\chi}|\widetilde T^{\mathrm{SK}}_{N,D,L,R}-\widehat Z_N|^2
=\E_{J,\chi}|\widetilde T^{\mathrm{SK}}_{N,D,L,R}-T^{\mathrm{SK}}_{N,D,L}|^2
+\E_J|T^{\mathrm{SK}}_{N,D,L}-\widehat Z_N|^2.
\end{align*}
Combining Propositions~\ref{prop:sk-truncation} and~\ref{prop:sk-estimator} gives \eqref{eq:sk-complete-error}, where the first three terms are truncation errors, while the last term is the error from random coloring.

For the operation count, fix a candidate set~$A$ of size~$b$ and one trial. Proposition~\ref{prop:sk-colorful} bounds the computation for a fixed coloring by~$O((N+2L+1)^6 16^L36^b)$. Computing inclusion probabilities, drawing colors, and accumulating the corrected entries fit the same bound. The candidate sets can be generated in~$O((N+1)^2\sum_{b\leq D}\binom Nb)$ operations as in Proposition~\ref{prop:pure-evaluation}. Summing the bound for one trial over these sets and the~$R$ trials, one obtains~\eqref{eq:sk-evaluation-cost}. 
\end{proof}

\section{Proof of the main results}\label{sec:error-complexity}

This section, by specifying the relevant constants, turns the truncation estimates into mean square approximation for $\hat{Z}_N$, combining with  a lower tail bound for $\hat{Z}_N$, we prove Theorem~\ref{thm:finite-size}.

\subsection{Choosing the cutoffs}\label{subsec:cutoff-selection}

Given a requested mean-square error~$u\in(0,1)$, we use~\eqref{eq:higher-truncation} and~\eqref{eq:sk-complete-error} to choose parameters that make the corresponding error bound at most~$u$. We also need the evaluation cost to be polynomial in~$N$ and~$u^{-1}$. Let~$k$ denote the cutoff on used vertices for higher-order mixtures and on branching vertices for SK, where we take~$D=k$. The number of edges~$L$ and colorings~$R$ in SK case will be chosen afterwards in Proposition~\ref{prop:mean-square}.

The dependence on~$k$ in these error bounds is through~$e^{-\Phi_N(k+1)}$. Recall
\begin{equation}\label{eq:algorithm-support-scale}
\Phi_N(t)=\kappa t\log\frac{N}{C_\star t}\quad(t>0),
\qquad \Phi_N(0)=0,
\end{equation}
with the constants from Section~\ref{subsec:higher-support} for higher-order mixtures and Section~\ref{subsec:sk-support} for SK. For higher-order mixtures, $\kappa$ may depend on~$N,h$, but~$\kappa\geq\kappa_{\min}>0$. For SK, $\kappa=\kappa_{\min}=1$.

It suffices to bound~$e^{-\Phi_N(k+1)}$, the powers of~$\theta$ in~\eqref{eq:higher-truncation} and~\eqref{eq:sk-complete-error}, and, for SK, $e^{2L}/R$ by~$u/C_\Lambda$, where
\begin{equation}\label{eq:algorithm-cutoff-constants}
C_\Lambda=\begin{cases}
8(2+C_+),&p_m\geq3,\\
8(2C_{\mathrm{SK}}+3C_+),&\text{SK},
\end{cases}
\end{equation}
the two mean-square bounds are then at most~$(2+C_+)u/C_\Lambda=u/8$ and~$(2C_{\mathrm{SK}}+3C_+)u/C_\Lambda=u/8$, respectively. We first choose the cutoff \(k\) so that \(e^{-\Phi_N(k+1)}\le u/C_\Lambda\).


\begin{lemma}[Choosing the vertex cutoff]\label{lem:cutoff-selection}
There exist constants~$\gamma\in(0,1/4)$, $\nu>0$, and an integer~$N_0\geq P$, depending only on the fixed model and~$\beta_0,\rho_0$ such that the following holds. For every~$N\geq N_0$ and~$C_\Lambda e^{-\nu N}<u<1$, one can choose an integer~$0\leq k<\lfloor\gamma N\rfloor$ satisfying the truncation requirement~$e^{-\Phi_N(k+1)}\leq u/C_\Lambda$ and the enumeration bound
\begin{equation}\label{eq:cutoff-cost}
\sum_{j=0}^{k}\binom Nj t^j\leq(N+1)\left(\frac{C_\Lambda}{u}\right)^{C_t}
\end{equation}
for every fixed~$t\geq1$. Here~$C_t>0$ depends only on the fixed model, $\beta_0,\rho_0$, and~$t$.
\end{lemma}
\begin{proof}
Choose
\begin{equation*}
\gamma=\min\left\{\frac18,\frac{e^{-(1+2/\kappa_{\min})}}{C_\star}\right\}
\end{equation*}
and~$N_0=\lceil\max\{P,4,2/\gamma\}\rceil$. For~$0<x\leq\gamma N$, we have~$\Phi_N'(x)=\kappa(\log(N/(C_\star x))-1)\geq\kappa_{\min}(\log(1/(C_\star\gamma))-1)\geq2$. Thus~$\Phi_N$ is increasing on~$(0,\gamma N]$. For~$N\geq N_0$, the inequality~$\lfloor\gamma N\rfloor\geq\gamma N/2$ gives
\begin{equation*}
\Phi_N(\lfloor\gamma N\rfloor)
\geq\Phi_N(\gamma N/2)
\geq\frac{\kappa_{\min}\gamma N}{2}\log\frac{2}{C_\star\gamma}.
\end{equation*}

The terms~$C_+\theta^{2\gamma N/P}$ and~$C_+\theta^{3\gamma N}$ in~\eqref{eq:higher-truncation} and~\eqref{eq:sk-complete-error}, respectively, equal~$C_+e^{-c_1N}$, where
\begin{equation*}
c_1=\begin{cases}
(\gamma/P)\log(1/\theta^2),&p_m\geq3,\\
(3\gamma/2)\log(1/\theta^2),&\text{SK}.
\end{cases}
\end{equation*}
Set
\begin{equation}\label{eq:algorithm-remainder-scale}
\nu=\frac14\min\left\{1,\ c_1,\ \frac{\kappa_{\min}\gamma}{2}\log\frac{2}{C_\star\gamma}\right\}.
\end{equation}
Let~$N\geq N_0$ and~$C_\Lambda e^{-\nu N}<u<1$. With~$\Lambda=\log(C_\Lambda/u)$, we have~$0<\Lambda<\nu N$ and~$\Phi_N(\lfloor\gamma N\rfloor)\geq4\nu N>\Lambda$. Consequently, there is an integer~$0\leq k<\lfloor\gamma N\rfloor$ satisfying~$\Phi_N(k+1)\geq\Lambda$. Choose the least such integer. Then~$e^{-\Phi_N(k+1)}\leq e^{-\Lambda}=u/C_\Lambda$. Minimality gives~$\Phi_N(k)<\Lambda$ for~$k\geq1$. The same inequality holds for~$k=0$, since~$\Phi_N(0)=0$. The choice of~$\nu$ also gives~$c_1N\geq4\nu N>4\Lambda$, so
\begin{equation}\label{eq:cutoff-error}
e^{-c_1N}\leq e^{-4\Lambda}.
\end{equation}

Fix~$t\geq1$. If~$k=0$, the sum in~\eqref{eq:cutoff-cost} equals one, so the claimed bound holds. Suppose~$k\geq1$. The bound~$\binom Nj\leq(eN/j)^j$ for~$j\geq1$ and the monotonicity of~$x\log(etN/x)$ on~$(0,N]$ give
\begin{equation*}
\sum_{j=0}^{k}\binom Nj t^j\leq(k+1)\exp\left(k\log\frac{etN}{k}\right).
\end{equation*}
Since~$k<\gamma N$ and~$\kappa\geq\kappa_{\min}$, we have~$\Phi_N(k)\geq\kappa_{\min}k\log(1/(C_\star\gamma))>0$. Define~$C_t=\kappa_{\min}^{-1}\bigl(1+\log(etC_\star)/\log(1/(C_\star\gamma))\bigr)$. Using this lower bound on~$\Phi_N(k)$ and~$C_\star\geq1$, we obtain
\begin{equation*}
k\log\frac{etN}{k}
=\frac{\Phi_N(k)}{\kappa}+k\log(etC_\star)
\leq C_t\Phi_N(k)<C_t\Lambda.
\end{equation*}
The enumeration sum is therefore at most~$(N+1)e^{C_t\Lambda}=(N+1)(C_\Lambda/u)^{C_t}$, proving~\eqref{eq:cutoff-cost}. The cutoff can be found by testing at most~$\lfloor\gamma N\rfloor$ successive integers, using~$O(N)$ operations.
\end{proof}

Fix the constants~$\gamma,\nu,N_0$ constructed in the proof, and write~$\Lambda=\log(C_\Lambda/u)$ for the rest of this subsection. For higher-order mixtures, we evaluate the retained sum exactly. For SK, we must also control the edge error~$C_+\theta^{2(L+1)}$ and the coloring error~$C_+e^{2L}/R$. Algorithm 4 specifies edge cutoff \(L\) and coloring trial cutoff\(R\). The algorithm uses these truncations when~$N\geq N_0$ and~$\Lambda<\nu N$. Otherwise it evaluates the sum exactly. The following Proposition verifies the resulting error and cost bounds.

\begin{algorithm}[!htbp]
\caption{\textsc{GraphApprox}$(N,\beta,h,J,u)$}\label{alg:graph-approx}
\begin{algorithmic}[1]
\Require Finite disorder array~$J$, $N\geq P$, parameters satisfying~\eqref{eq:field-range}, and~$u\in(0,1)$.
\State Compute~$\omega_e=\tanh x_e$ and set~$\Lambda\gets\log(C_\Lambda/u)$.
\If{$N<N_0$ or $\Lambda\geq\nu N$}
  \State \Return the exact spin average~$G([N])$ in~\eqref{eq:local-spin-average}.
\EndIf
\State Compute~$\kappa$ by~\eqref{eq:field-dependent-kappa} for higher-order mixtures, or set~$\kappa\gets1$ for SK.
\State Find the smallest~$0\leq k<\lfloor\gamma N\rfloor$ satisfying~$\Phi_N(k+1)\geq\Lambda$.
\If{$p_m\geq3$}
  \State \Return $\textsc{SupportSum}(N,h,\omega,k)$.
\Else
  \State $D\gets k$, $L\gets\lceil\Lambda/\log(1/\theta^2)\rceil$, and~$R\gets\lceil e^{2L+\Lambda}\rceil$.
  \State \Return $\textsc{SKSum}(N,h,\omega,D,L,R)$.
\EndIf
\end{algorithmic}
\end{algorithm}

\begin{proposition}\label{prop:mean-square}
Suppose the disorder variables are independent and have a common symmetric law with unit variance, as in~\eqref{eq:disorder}. For every~$N\geq P$, specified~$\beta\in[0,\beta_0]$, external field~$|h|\leq\rho_0N^{-1/4}$, and~$u\in(0,1)$, Algorithm~\ref{alg:graph-approx} returns~$T$ satisfying
\begin{equation}\label{eq:mean-square-guarantee}
\E|T-\widehat Z_N|^2\leq u.
\end{equation}
There is a constant~$C>0$, depending only on the fixed model and~$\beta_0,\rho_0$, such that the algorithm uses at most~$CN^Cu^{-C}$ operations on every finite input and every random outcome. It is deterministic for higher-order mixtures and randomized for SK.
\end{proposition}

\begin{proof}
Recall that \(\nu,N_0\) are the fixed constants constructed in Lemma~\ref{lem:cutoff-selection} and \(\Lambda=\log(C_\Lambda/u)\). First consider the case in which the above algorithm uses truncation, namely~$N\geq N_0$ and~$\Lambda<\nu N$. The latter condition is equivalent to~$u>C_\Lambda e^{-\nu N}$. The algorithm uses the cutoff constructed in the proof of Lemma~\ref{lem:cutoff-selection}. The cutoffs depend only on~$N,h,u$ and the fixed parameters, so the earlier mean-square estimates apply with cutoffs independent of the disorder.

For higher-order mixtures, Proposition~\ref{prop:pure-evaluation} evaluates the retained sum exactly. Proposition~\ref{prop:higher-truncation}, Lemma~\ref{lem:cutoff-selection}, and~\eqref{eq:cutoff-error} give
\begin{align*}
\E|T_{N,k}-\widehat Z_N|^2
\leq2e^{-\Lambda}+C_+e^{-4\Lambda}\leq(2+C_+)e^{-\Lambda}=\frac u8.
\end{align*}
The cost bound~\eqref{eq:pure-evaluation-cost} and Lemma~\ref{lem:cutoff-selection} with~$t=2$ give at most~$C_P(N+1)^{P+3}e^{C_2\Lambda}$ operations.

For SK, the choices of~$L$ and~$R$ ensure~$\theta^{2(L+1)}\leq e^{-\Lambda}$ and~$e^{2L}/R\leq e^{-\Lambda}$. The combined error estimate~\eqref{eq:sk-complete-error}, Lemma~\ref{lem:cutoff-selection}, and~\eqref{eq:cutoff-error} therefore give
\begin{align*}
\E\left|\widetilde T^{\mathrm{SK}}_{N,k,L,R}-\widehat Z_N\right|^2
&\leq(2C_{\mathrm{SK}}+2C_+)e^{-\Lambda}+C_+e^{-4\Lambda}\\
&\leq(2C_{\mathrm{SK}}+3C_+)e^{-\Lambda}=\frac u8.
\end{align*}
By~\eqref{eq:sk-evaluation-cost} and Lemma~\ref{lem:cutoff-selection} with~$t=36$, the cost is at most
\begin{equation*}
C_{\mathrm{alg}}R(N+2L+1)^6\,16^L(N+1)e^{C_{36}\Lambda}.
\end{equation*}
Since~$L\leq1+\Lambda/\log(1/\theta^2)$ and~$\Lambda<\nu N$, we have~$N+2L+1\leq C(N+1)$. Also, $R\leq2e^{2L+\Lambda}$ gives
\begin{align*}
R16^Le^{C_{36}\Lambda}
\leq2\exp\left((2+\log16)L+(1+C_{36})\Lambda\right)
\leq Ce^{C'\Lambda}
\end{align*}
for fixed constants~$C,C'>0$. The SK cost is consequently at most~$C(N+1)^7e^{C'\Lambda}$. In both models, the identity~$e^\Lambda=C_\Lambda/u$ converts these estimates into polynomial bounds in~$N$ and~$u^{-1}$.

It remains to consider exact enumeration. Here~$G([N])=\widehat Z_N$, so the error is zero. Its cost is at most~$C_P(N+1)^{P+1}2^N$. If~$N<N_0$, then~$2^N$ is bounded by a fixed constant. Otherwise~$\Lambda\geq\nu N$, and
\begin{equation}\label{eq:exact-enumeration-cost}
2^N\leq e^{(\log2/\nu)\Lambda}
=\left(\frac{C_\Lambda}{u}\right)^{\log2/\nu}.
\end{equation}
Thus exact enumeration also satisfies the claimed polynomial bound. Computing the input weights and searching for~$k$ add only polynomial work in~$N$. Computing the SK integer ceilings by comparisons and increments costs~$O(L+R)$, which is covered by the preceding estimates. 
\end{proof}

\subsection{Lower tail bound for $\hat{Z}_N$}\label{subsec:lower-tail}

To obtain logarithmic accuracy, we need to control the probability that~$\widehat Z_N$ is close to zero. On a suitable event determined by the coupling magnitudes, we will express~$\widehat Z_N$ as a noise average of a positive function of the independent coupling signs, then use reverse hypercontractivity to derive the lower tail bound.

We first recall the reverse hypercontractive inequality of Mossel,
Oleszkiewicz, and Sen~[27]. Let $(\Omega_i,\mu_i)$,
$1\le i\le m$, be finite probability spaces, and equip
$\Omega=\prod_{i=1}^m\Omega_i$ with the product measure
$\mu=\bigotimes_{i=1}^m\mu_i$. For a function on $\Omega$, let
$\mathbb{E}_i$ denote averaging over its $i$th coordinate while keeping
the other coordinates fixed. Consider the operators
\[
T_t
=
\prod_{i=1}^{m}
\left[e^{-t}\mathrm{Id}+(1-e^{-t})\mathbb{E}_i\right],
\qquad t\ge0.
\]
Thus for each $f:\Omega \to \mathbb{R}$, $T_t f(x)$ is obtained by independently retaining each coordinate
of $x$ with probability $e^{-t}$, and otherwise resampling that
coordinate from $\mu_i$, before evaluating $f$ and taking the
expectation. To be precise, 
\[
T_t f(x) = \sum_{y\in \Omega} f(y) \prod_{i=1}^m[e^{-t} 1_{y_i = x_i} + (1-e^{-t}) \mu_i(y_i)].
\]
These operators are precisely the tensor product of the simple
Markov semigroups considered in~\cite{MOS13}. For a strictly positive function $f$, denote
\(
\|f\|_r=\bigl(\mathbb{E}_\mu f^r\bigr)^{1/r}
\)
for $r \neq 0$ and $\|f\|_0=\exp(\mathbb{E}_\mu\log f)$. We restate the Corollary 1.11 of \cite{MOS13}.

\begin{lemma}[\cite{MOS13}*{Corollary 1.11}]\label{lem:MOS}
    For every strictly positive function $f$ on $\Omega$ and every pair of exponents
$q<p<1$,
\[
\|T_t f\|_q\ge\|f\|_p
\qquad\text{whenever}\qquad
t\ge\log\frac{1-q}{1-p}.
\]
\end{lemma}

We apply this result to the uniform measure on the sign cube
$\Omega=\{-1,1\}^m$, and write $\mathbb{E}$ for expectation under
this measure. For $0<\vartheta<1$, define
\[
(\mathsf{N}_\vartheta f)(\epsilon)
=
2^{-m}\sum_{\eta\in\{-1,1\}^m}
f(\eta)\prod_{i=1}^{m}(1+\vartheta\epsilon_i\eta_i).
\]
This is precisely the preceding operator $T_t$ with
$e^{-t}=\vartheta$, equivalently, $\mathsf{N}_\vartheta=T_{\log(1/\vartheta)}$.


The action of $\mathsf{N}_\vartheta$ on sign monomials is particularly
simple. For $B\subseteq[m]$, let
$w_B(\epsilon)=\prod_{i\in B}\epsilon_i$, we have \(
\mathsf{N}_\vartheta w_B=\vartheta^{|B|}w_B
\) using the following identities
\[
\frac12\sum_{\eta_i=\pm1}
(1+\vartheta\epsilon_i\eta_i)=1, \qquad 
\frac12\sum_{\eta_i=\pm1}
\eta_i(1+\vartheta\epsilon_i\eta_i)
=
\vartheta\epsilon_i.
\]

We will use this identity to express the normalized partition function
as the noise average of a function with rescaled interaction weights. Before that, we state the following result as a direct application of the above
lemma. 

\begin{lemma}\label{lem:reverse-hypercontractivity}
Let~$0<\vartheta<1$ and~$s=\vartheta^{-1/2}-1$. Every function~$f:\Omega\to(0,\infty)$ with~$\E f=1$ satisfies
\begin{equation}\label{eq:reverse-hypercontractivity}
\E\bigl[(\mathsf N_\vartheta f)^{-s}\bigr]\leq\E f^2.
\end{equation}
\end{lemma}
\begin{proof}
Set $q= -s = 1- \vartheta^{-1/2}$ and $p = 1- \sqrt{\vartheta}$, since $\mathsf N_\vartheta = T_{\log (1/\vartheta)}$, applying Lemma~\ref{lem:MOS} gives
\begin{equation*}
\E\bigl[(\mathsf N_\vartheta f)^{-s}\bigr]
\leq(\E f^p)^{-s/p}.
\end{equation*}
H\"older's inequality gives
\begin{equation*}
1=\E f
=\E\!\left[(f^p)^{1/(2-p)}(f^2)^{(1-p)/(2-p)}\right]
\leq(\E f^p)^{1/(2-p)}(\E f^2)^{(1-p)/(2-p)}.
\end{equation*}
Hence~$\E f^p\geq(\E f^2)^{-(1-p)}$, and 
\(
\E\bigl[(\mathsf N_\vartheta f)^{-s}\bigr]
\leq(\E f^2)^{s(1-p)/p}=\E f^2
\)
since $s = p/(1-p)$.
\end{proof}

For the normalized partition function $\hat{Z}_N$, the cube has one coordinate per interaction, so~$m=|\mathcal E_N|$. We will use~$\vartheta=\theta=\beta_0/\beta_+\in(0,1)$ from~\eqref{eq:algorithm-temperature-constants}. Define the exponent~$s$ in Algorithm~\ref{alg:log-approx} and the tail constant~$c$ in Theorem~\ref{thm:finite-size} by
\begin{equation}\label{eq:lower-tail-constants}
s=\theta^{-1/2}-1>0,\qquad c=\frac{\operatorname{arctanh}(\theta/2)}{\beta_0}.
\end{equation}
The event
\begin{equation}\label{eq:bounded-coupling-event}
\mathcal B_N=\bigcap_{p\in\mathcal P}
\left\{\max_{e\in\mathcal E_{N,p}}|a_pJ_e|\leq cN^{(p-1)/2}\right\}
\end{equation}
implies~$|x_e|\leq\beta_0c$ and hence~$|\omega_e|\leq\tanh(\beta_0c)=\theta/2$ for every edge. A union bound gives~$\pr(\mathcal B_N^c)\leq\tau_N$. This event will ensure positivity of the function to which we apply Lemma~\ref{lem:reverse-hypercontractivity}. It is used only in the proof, and the algorithm does not modify the disorder.

\begin{proposition}\label{prop:lower-tail}
For every~$N\geq P$, specified~$\beta\in[0,\beta_0]$, and~$|h|\leq\rho_0N^{-1/4}$, we have 
\begin{equation}\label{eq:lower-tail}
 \pr(\widehat Z_N<z)\leq C_+z^s+\tau_N, 
 \quad \text{for $z >0$}.
\end{equation}
where $s$ is defined above in~\eqref{eq:lower-tail-constants} and $\tau_N$ is the disorder tail bound.
\end{proposition}
\begin{proof}
By symmetry, represent the disorder in law as~$J_e=\epsilon_e|J_e|$, where the~$\epsilon_e$ are independent uniform signs, independent of the magnitudes, then~$\omega_e=\epsilon_e|\omega_e|$. Fix the magnitudes and let~$\E_\epsilon$ denote uniform expectation on~$\{-1,1\}^{\mathcal E_N}$. The event~$\mathcal B_N$ depends only on the magnitudes, so restricting to this event does not change the conditional sign law.

To construct a function whose noise average is~$\widehat Z_N$, we divide each edge weight by~$\theta$ to cancel the factor~$\theta^{|\Gamma|}$ produced by the noise. Define
\begin{equation}\label{eq:inflated-sign-function}
f(\epsilon)=\E_h\prod_e\left(1+\frac{|\omega_e|}{\theta}\epsilon_e\sigma_e\right)
=\sum_{\Gamma\subseteq\mathcal E_N}\hh^{|\partial\Gamma|}
\prod_{e\in\Gamma}\frac{|\omega_e|\epsilon_e}{\theta}.
\end{equation}
Here~$\E_h$ averages the spins with~$\epsilon$ fixed. On~$\mathcal B_N$, every factor $1+\frac{|\omega_e|}{\theta}\eps_e \sigma_e$ in the product is at least~$1/2$, so~$f>0$. Averaging over the signs leaves only the empty graph, giving~$\E_\epsilon f=1$. The monomial identity above yields
\begin{equation*}
(\mathsf N_\theta f)(\epsilon)
=\sum_{\Gamma\subseteq\mathcal E_N}\hh^{|\partial\Gamma|}
\prod_{e\in\Gamma}|\omega_e|\epsilon_e
=\widehat Z_N(\epsilon).
\end{equation*}
Lemma~\ref{lem:reverse-hypercontractivity} therefore gives, on~$\mathcal B_N$,
\(
\E_\epsilon\widehat Z_N^{-s}\leq\E_\epsilon f^2.
\)
Orthogonality of the sign monomials gives, for every magnitude array,
\begin{equation}\label{eq:conditional-second-moment}
\E_\epsilon f^2=\sum_{\Gamma\subseteq\mathcal E_N}
\hh^{2|\partial\Gamma|}\prod_{e\in\Gamma}\frac{|\omega_e|^2}{\theta^2}.
\end{equation}
Averaging over the magnitudes, whose expectation we denote by~$\E_{|J|}$, and using~$\E_\epsilon f^2\geq0$, we obtain
\begin{align*}
\E[\widehat Z_N^{-s}\ind_{\mathcal B_N}]
=\E_{|J|}\!\left[\ind_{\mathcal B_N}\E_\epsilon\widehat Z_N^{-s}\right]
\leq\E_{|J|}[\ind_{\mathcal B_N}\E_\epsilon f^2]
\leq\E_{|J|}\E_\epsilon f^2
=\sum_{\Gamma\subseteq\mathcal E_N}\hh^{2|\partial\Gamma|}
\prod_{e\in\Gamma}\frac{r_{N,|e|}}{\theta^2}
\leq C_+.
\end{align*}
For the last inequality, $r_{N,p}/\theta^2\leq\bar r_{N,p}$, and Proposition~\ref{prop:mass} bounds the sum with these comparison weights. Markov's inequality on~$\mathcal B_N$ and~$\pr(\mathcal B_N^c)\leq\tau_N$ give
\begin{equation*}
\pr(\widehat Z_N<z)
\leq\pr(\mathcal B_N^c)+z^s\E[\widehat Z_N^{-s}\ind_{\mathcal B_N}]
\leq\tau_N+C_+z^s.
\end{equation*}
\end{proof}

\subsection{Proof of the main theorem}\label{subsec:main-proofs}

It remains to combine the mean-square and lower-tail estimates with Algorithm~\ref{alg:log-approx}. 

\begin{proof}[Proof of Theorem~\ref{thm:finite-size}]
Let~$z,u$ be as in Algorithm~\ref{alg:log-approx}, and let~$T$ be its output from~\textsc{GraphApprox}. Since $0<z<1$ and~$0<u<1/8$, so Proposition~\ref{prop:mean-square} applies. We first identify an event on which~$T$ is positive and its logarithm has the required accuracy. On the event
\begin{equation*}
\widehat Z_N\geq z,\qquad |T-\widehat Z_N|\leq\varepsilon z/2,
\end{equation*}
we have~$T\geq z(1-\varepsilon/2)>0$ and~$|T/\widehat Z_N-1|\leq\varepsilon/2$. The bound~$|\log(1+x)|\leq2|x|$ for~$|x|\leq1/2$ gives~$|\log T-\log\widehat Z_N|\leq\varepsilon$. The other factors in~$Z_N$ are computed exactly, so the output has the same error for~$\log Z_N$.

We now bound the probability that this event fails. Proposition~\ref{prop:lower-tail} controls small values of~$\widehat Z_N$, and Markov's inequality controls the absolute error of~$T$. In Algorithm~\ref{alg:log-approx}, we have~$C_+z^s=\delta/2$ and~$4u/(\varepsilon^2z^2)=\delta/2$. Consequently,
\begin{align*}
\pr(|\widehat L_N-\log Z_N|>\varepsilon)
&\leq\pr(\widehat Z_N<z)+\pr(|T-\widehat Z_N|>\varepsilon z/2)\\
&\leq\tau_N+C_+z^s+\frac{4\E|T-\widehat Z_N|^2}{\varepsilon^2z^2}
\leq\tau_N+\delta/2+\delta/2.
\end{align*}
 The fallback when~$T\leq0$ occurs only outside the event above and is covered by this bound. To conclude, substitute~$u^{-1}=8(2C_+)^{2/s}\varepsilon^{-2}\delta^{-(1+2/s)}$ into the cost bound of Proposition~\ref{prop:mean-square}. This gives~$CN^C\varepsilon^{-C}\delta^{-C}$. Computing~$\log\overline Z_N$ and the other explicit terms adds only polynomial work in~$N$. All constants have the stated dependence.
\end{proof}

\bibliography{references}

\end{document}